\documentclass[12pt,letterpaper]{article} 
\usepackage{lipsum}
\usepackage{url}
\usepackage{enumitem}
\usepackage[nochapters]{classicthesis} 
\usepackage{graphicx}

\usepackage{geometry}
 \usepackage{authblk}
\usepackage{blindtext}
 
 \usepackage{titlesec}

\titleformat{\subsubsection}[runin]
       {\normalfont\bfseries}
       {\thesubsubsection}
       {0.5em}
       {}
       [.]

\titleformat{\subsection}[runin]
       {\normalfont\bfseries}
       {\thesubsection}
       {0.5em}
       {}
       [.]
       
\usepackage[T1]{fontenc}
\usepackage[utf8]{inputenc}
\usepackage[american]{babel}
\usepackage{amsmath, amsthm, amsfonts, amssymb, slashed, stmaryrd, mathrsfs}

\theoremstyle{plain}

\newtheorem{Lemma}{Lemma}[section]
\newtheorem{Corollary}{Corollary}[section]
\newtheorem{Theorem}{Theorem}[section]
\newtheorem{Proposition}{Proposition}[section]
\newtheorem*{Theorem*}{Theorem}
\newtheorem*{Claim*}{Claim}
\newtheorem*{Proposition*}{Proposition}
\newtheorem{Question}{Question}[section]

\theoremstyle{definition}
\newtheorem{Example}{Example}[section]
\newtheorem{Non-Example}{Non-Example}[section]
\newtheorem{Definition}{Definition}[section]

\theoremstyle{remark}

\newtheorem{Remark}{Remark}[section]

\usepackage{tikz-cd}
\usepackage{aliascnt, hyperref}

\usepackage{scalerel,stackengine}
\stackMath
\newcommand\widecheck[1]{%
\savestack{\tmpbox}{\stretchto{%
  \scaleto{%
    \scalerel*[\widthof{\ensuremath{#1}}]{\kern-.6pt\bigwedge\kern-.6pt}%
    {\rule[-\textheight/2]{1ex}{\textheight}}
  }{\textheight}%
}{0.5ex}}%
\stackon[1pt]{#1}{\scalebox{-1}{\tmpbox}}%
}

\usepackage{stmaryrd}

\date{\vspace{-5ex}}

\usepackage[
backend=biber,
style=alphabetic,
sorting=nyt, doi=false, isbn=false, url=false, eprint=false
]{biblatex}
\newcommand{\HMto}{\widecheck{\mathrm{HM}}}

\newcommand{\HMtilde}{\widetilde{\mathrm{HM}}}

\newcommand{\fc}{\mathbf{Fc}}
\newcommand{\fctilde}{\widetilde{\mathbf{Fc}}}

\newcommand{\cto}{\mathbf{c}}

\begin{document}

\title{Exotic Dehn twists on sums of two contact $3$-manifolds}

\author[1]{Eduardo Fernández \thanks{eduardofernandez@uga.edu}}
\author[1]{Juan Muñoz-Echániz \thanks{juanmunoz@math.columbia.edu}}

\affil[1]{Department of Mathematics, University of Georgia}
\affil[1]{Department of Mathematics, Columbia University}

\maketitle




\abstract{We exhibit the first examples of exotic contactomorphisms with infinite order as elements of the contact mapping class group. These are given by certain Dehn twists on the separating sphere in a connected sum of two closed contact $3$-manifolds. We detect these by a combination of hard and soft techniques. On the one hand, we make essential use of an invariant for families of contact structures which generalises the Kronheimer--Mrowka contact invariant in monopole Floer homology. We then exploit an $h$-principle for families of convex spheres in tight contact $3$-manifolds, from which we establish a parametric version of Colin´s decomposition theorem. As a further application, we also exhibit new exotic $1$-parametric phenomena in overtwisted contact $3$-manifolds.}

\tableofcontents

\section{Introduction}

Throughout this article all $3$-manifolds are closed, oriented and connected unless otherwise noted, and all contact structures on $3$-manifolds are co-oriented and positive. 

\subsection{Main result} 

A fundamental problem in contact topology is to understand the isotopy classes of contact diffeomorphisms, usually called "contactomorphisms", of a contact manifold. The following is a longstanding open question in all dimensions: 

\begin{Question}
Do there exist exotic contactomorphisms with infinite order as elements in the contact mapping class group?
\end{Question}

In this article we answer this question in the \textit{affirmative} in dimension \textit{three}. Here, and throughout the article, by \em exotic \em we will mean \em non-trivial in the contact category \em but \em formally trivial \em (and, in particular, \em trivial in the smooth category \em). See \S \ref{exotic} and below for further details. We consider a contact $3$-manifold given by the connected sum of two contact $3$-manifolds $(Y_\#, \xi_\#) := (Y_- , \xi_-) \# (Y_+ , \xi_+ )$. Recall that the connected sum is built by removing Darboux balls $B_\pm \subset Y_\pm$ and gluing the complements $Y \setminus B_\pm$ by an orientation-reversing diffeomorphism of their boundary spheres which preserves their characteristic foliations. Reparametrisation of one of the spheres provides a $\mathrm{U}(1)$ worth of choices for gluing, and thus $(Y_\#, \xi_\# )$ naturally belongs in a \textit{family} of contact $3$-manifolds
\begin{align*}
    (Y_\# , \xi_\# ) \hookrightarrow \mathcal{Y}_\# \rightarrow \mathrm{U}(1).
\end{align*}
The monodromy of this family is realised by a contactomorphism of $(Y_\# , \xi_\# )$,  well-defined up to contact isotopy. Its underlying diffeomorphism is the \textit{Dehn twist} on the separating sphere $S_\#$ in the neck of the connected sum $Y_\# = Y_- \# Y_+$. We denote this contactomorphism $\tau_{S_\#}$ and call it the \textit{contact Dehn twist} on $S_\#$. Unlike previous constructions of contactomorphisms, the contact Dehn twist is a \textit{local symmetry} of an arbitrarily small neighbourhood of a $2$-sphere (see \S \ref{dehnsection} for further details). As a diffeomorphism, the Dehn twist can be isotoped so that it is supported on a neighborhood $[0,1]\times S^2$ of $S_\#\simeq S^2$ on which it acts as $[0,1]\times S^2\ni(t,p)\mapsto (t, R_{\theta(t)}(p))$, where $R_\varphi$ denotes the rotation of angle $\varphi$ along the $z$ axis in $\mathbb{R}^3$, and $\theta:[0,1]\rightarrow [0,2\pi]$ is a smooth function with $\theta\equiv 0$ near $t=0$ and $\theta\equiv 2\pi$ near $t=1$. Because $\pi_1 \mathrm{SO}(3) = \mathbb{Z}/2$ we have that the $2$-fold iterate $\tau_{S_\#}^2$ is \textit{smoothly} isotopic to the identity, but it remains to be understood whether

\begin{Question}
Is $\tau_{S_\#}^2$ contact isotopic to the identity?
\end{Question}

Associated to the contact structures $\xi_\pm$ we have their Kronheimer--Mrokwa contact invariants $\mathbf{c}(\xi_\pm ) \in \HMto (-Y_\pm )$ \cite{monocont}\cite{monolens}. These are canonical elements (defined up to sign) in the "to" flavor of the monopole Floer homology of $- Y_\pm$. The contact invariant was also defined in the setting of Heegaard-Floer homology by Ozsváth and Szabó \cite{OScontact}. Under the isomorphism between the monopole and Heegaard-Floer homologies \cite{KLT}\cite{CGH} the contact invariants agree. Throughout this article we only consider monopole Floer homology and the contact invariant with \textit{coefficients} in $\mathbb{Q}$, for simplicity. The main result of this article is the following

\newpage

\begin{Theorem}\label{mainthm}
Let $(Y_\pm , \xi_\pm )$ be \textit{irreducible} contact $3$-manifolds. Suppose that the Kronheimer--Mrowka contact invariants $\mathbf{c} (\xi_{\pm}  )$ do not lie in the image of the $U$-map $$U : \widecheck{\mathrm{HM}}(-Y_{\pm} ) \rightarrow \widecheck{\mathrm{HM}}(-Y_{\pm} ).$$ Then
\begin{enumerate}[label=(\Alph*)]
\item The $k$-fold iterates $\tau_{S_\#}^k$, $k \geq 1$, of the contact Dehn twist are not contact isotopic to the identity.
\item If the Euler classes of $\xi_{\pm}$ vanish, then $\tau_{S_\#}^2$ is formally contact isotopic to the identity.
\end{enumerate}
\end{Theorem}

We now explain the meaning of the assertion in Theorem \ref{mainthm}(B). Given a contact $3$-manifold $(Y,\xi)$, a \textit{formal contactomorphism} of $(Y, \xi )$ consists of a pair $(f , F )$ where $f$ is a diffeomorphism of $Y$ and $F = (F_s)$ is a homotopy through vector bundle isomorphisms $F_s : TY \rightarrow f^\ast TY$ such that $F_0  = d f$ and $F_1$ preserves $\xi$. Any contactomorphism $f$ yields a formal contactomorphism, and one says that $f$ is \textit{formally trivial} if $f$ can be deformed to the identity through formal contactomorphisms. A contactomorphism $f$ of $(Y, \xi )$ will be called \textit{exotic} if it is formally contact isotopic to the identity but is not contact isotopic to the identity. Thus, exotic contactomorphisms are those which are "geometrically" non-trivial, and not for reasons having to do with the underlying smooth or tangential structures. See \S \ref{exotic} for further context. Thus, Theorem \ref{mainthm} asserts that $\tau_{S_\#}^2$ and all its iterates are exotic.

\begin{Remark}In fact, we will establish more: the contactomorphism $\tau_{S_\#}^2$ from Theorem \ref{mainthm} has infinite order as an element in the \textit{abelianisation} of the group
\begin{align}
\mathrm{Ker} \big( \pi_0 \mathrm{Cont}(Y , \xi ) \rightarrow \pi_0 \mathrm{Diff}(Y ) \big). \label{Cont0}
\end{align}
\end{Remark}

\begin{Remark}
For comparison with Theorem \ref{mainthm}, whenever either of $(Y_{\pm} , \xi_{\pm})$ is the tight $S^1 \times S^2$ or a quotient of tight $(S^3 , \xi )$ (e.g. the lens spaces $L(p,q)$ or the Poincaré sphere $\Sigma (2,3,5)$) then the squared contact Dehn twist $\tau_{S_\#}^2$ of $(Y_\# , \xi_\# )$ is contact isotopic to the identity, see Lemmas \ref{quotients}-\ref{S1S2}.
\end{Remark}

\
We also establish an analogous result for connected sums with multiple summands. Let $(Y, \xi )$ be a \textit{tight} $3$-manifold. By the Prime Decomposition Theorem combined with Colin's Decomposition Theorem \cite{colin} (see also \cite{honda,DGsum}) we have a unique connected sum decomposition $$(Y, \xi ) \cong (Y_0 , \xi_0 ) \# \cdots \# (Y_N , \xi_N ) $$into tight contact $3$-manifolds $(Y_j , \xi_j )$, where each piece $Y_j$ is a prime $3$-manifold. Let $n+1 \leq N$ be the number of prime summands $(Y_j, \xi_j )$ such that $\mathbf{c}(\xi_j ) \notin \mathrm{Im}U$ and the Euler class of $\xi_j$ vanishes. Let $\mathcal{C}(Y, \xi )$ (resp. $\Xi (Y, \xi )$) be the space of contact structures (resp. co-oriented $2$-plane fields) on $Y$ in the path-component of $\xi$. 

\begin{Theorem}\label{mainthm2}
With $(Y, \xi )$ as above, when $n \geq 1$ there is a $\mathbb{Z}^{n}$ subgroup in the kernel of 
$$\pi_1  \mathcal{C}(Y , \xi )  \rightarrow \pi_1 \Xi (Y , \xi )  $$
which induces a $\mathbb{Z}^{n}$ subgroup in the first singular homology $\mathrm{H}_1 \big( \mathcal{C}(Y, \xi )  ; \mathbb{Z} \big)$.
\end{Theorem}

In particular, the exotic subgroup $\mathbb{Z}^n$ exhibited in Theorem \ref{mainthm2} can be arbitrarily large in the following sense: for every $n\geq 1$ there exists a tight contact $3$-manifold, in fact infinitely many, such that the kernel of the previous homomorphism contains a subgroup isomorphic to $\mathbb{Z}^n$. 

\begin{Remark} 
The $n$ homologically independent loops of contact structures that we detect in Theorem \ref{mainthm2} yield under the natural map 
$$\pi_1 \mathcal{C}(Y , \xi ) \rightarrow \pi_0 \mathrm{Cont}(Y , \xi )$$
the squared contact Dehn twists on each of the $n$ spheres which separate the $n+1$ prime summands $(Y_j , \xi_j )$. However, we are unable to establish that the corresponding squared contact Dehn twists are non-trivial or that they yield a subgroup $\mathbb{Z}^n \subset \pi_0 \mathrm{Cont}(Y, \xi )$ when $n \geq 2$, but we conjecture that this should be true. See Remark \ref{why}.
\end{Remark}

The proofs of Theorems \ref{mainthm} and \ref{mainthm2} combine rigid obstructions arising from Floer homology together with flexibility results. On the one hand, an essential ingredient is a families generalisation of the Kronheimer--Mrowka contact invariant in monopole Floer homology, introduced by the second author \cite{yo}. This obstructs the existence of sections of a natural fibration given by the \textit{evaluation map} $ev : \mathcal{C}(Y, \xi ) \rightarrow S^2$ which sends a contact structure to its plane at $p$, where $p \in Y$ is some fixed point. We combine this machinery with the multi-parametric convex surface theory techniques introduced by the first author together with J. Martínez-Aguinaga and F. Presas \cite{fmp}. In particular, we use these techniques to establish the following generalisation of the much celebrated Decomposition Theorem of Colin \cite{colin} which could be of independent interest for contact topologists, and which will be crucial to the proof of Theorem \ref{mainthm2}. 

We consider two tight contact $3$-manifolds $(Y_\pm , \xi_\pm )$ equipped with Darboux balls $B_\pm \subset (Y_\pm , \xi_\pm )$. Let $\mathcal{C}(Y_\pm , \xi_\pm , B_\pm ) \subset \mathcal{C}(Y_\pm , \xi_\pm )$ denote the subspace of contact structures on $Y_\pm$ that coincide with $\xi_\pm$ over $B_\pm$. We consider the evaluation maps $ev_\pm : \mathcal{C}(Y_\pm, \xi_\pm ) \rightarrow S^2$ which send a contact structure to its plane at the point $p_\pm$ given by the center of $B_\pm$. These maps are fibrations, and the inclusion of $\mathcal{C}(Y_\pm, \xi_\pm , B_\pm )$ into the fiber of $ev_\pm$ induces a homotopy equivalence. We form the connected sum $(Y_\#, \xi_\#) = (Y_-  , \xi_-) \# (Y_+ , \xi_+ )$ by carving out the balls $B_\pm$ and gluing together the boundary components thus created. Consider the evaluation map $ev_\# : \mathcal{C}(Y_\#, \xi_\# ) \rightarrow S^2$ at a point on the "neck" region. We establish the following $h$-principle type result, which should be regarded as a parametric version of Colin's Theorem:

\begin{Theorem}\label{sumcontact2}
The inclusion of $\mathcal{C}(Y_- , \xi_- , B_- ) \times \mathcal{C}(Y_+ , \xi_+ , B_+)$ into the fiber of $ev_\#$ induces a homotopy equivalence. Thus, there is a fibration sequence
$$\mathcal{C}(Y_- , \xi_- , B_- ) \times \mathcal{C}(Y_+ , \xi_+ , B_+)\hookrightarrow \mathcal{C}(Y_\#,\xi_\#)\xrightarrow{ev_\#} S^2 .$$
\end{Theorem}  

We refer to Theorem \ref{thm:ConnectedSumTight} for a more general version. 

\subsection{Examples}

We now give examples of irreducible contact $3$-manifolds $(Y, \xi )$ such that $\mathbf{c}(\xi ) \notin \mathrm{Im}U$, many of which also have vanishing Euler class. 

\begin{Example}(Links of singularities)
The simplest example is the Brieskorn sphere $$\Sigma (p,q,r) = \big\{ (x,y,z) \in \mathbb{C}^3 \, | \, x^p + y^q + z^r = 0 \text{  and  } |x|^2 + |y|^2 + |z|^2 = \epsilon \big\}$$ where $\epsilon \in \mathbb{R}_{> 0}$ is small and $p,q,r \geq 1$ are integers with $1/p + 1/q + 1/r < 1$, equipped with the contact structure $\xi_{\mathrm{sing}}$ induced from the Brieskorn singularity. More generally, we could take any isolated normal surface singularity germ $(X, o )$ and let $(Y, \xi_{\mathrm{sing}} )$ be the contact manifold arising as the \textit{link} of the singularity. Neumann \cite{neumann} proved that the $3$-manifold $Y$ is irreducible. Provided that $Y$ is also a rational homology sphere, then the following are equivalent statements, as proved by Bodnár, Plamenevskaya \cite{plamenevskaya} and Némethi \cite{nemethi}:
\begin{enumerate}[label=(\alph*)]
    \item $\mathbf{c}(\xi_{\mathrm{sing}}  ) \notin \mathrm{Im} U$
    \item $Y$ is not an $L$-space
    \item $(X, o )$ is not a rational singularity.
\end{enumerate}
For instance, all Seifert fibered integral homology spheres excluding $S^3$ or the Poincaré sphere carry a contact structure $\xi_{\mathrm{sing}}$ with the above properties.
\end{Example}

\begin{Example}
Several surgeries on the Figure Eight knot are hyperbolic (hence irreducible) and support contact structures with $c(\xi ) \notin \mathrm{Im}U$. Contact structures on these manifolds have been classified by Conway and Min \cite{conway-min}. 
\end{Example}

\begin{Example}
 All but one of the $\frac{n(n-1)}{2}$ tight contact structures supported on $- \Sigma ( 2,3,6n-1)$ up to isotopy, classified by Ghiggini and Van Horn-Morris \cite{ghiggini-vanhorn-morris}.
\end{Example}

\subsection{Exotic overtwisted phenomena} 

Let $(Y, \xi )$ be such that $\mathbf{c}(\xi ) \notin \mathrm{Im}U$ and $\xi$ has vanishing Euler class. Let $B \subset (Y,\xi )$ be a Darboux ball. From this, one can produce overtwisted contact manifolds by modifying $(Y, \xi )$ by a Lutz Twist inside $B$, or by taking the connected sum (using $B$) with an overtwisted contact manifold $(M, \xi_{\mathrm{ot}})$. In either case, the squared contact Dehn twist on the boundary of $B$ becomes isotopic to the identity in this new overtwisted manifold, by an application of Eliashberg's \textit{h}-principle for overtwisted contact structures \cite{EliashbergOT}. However, this has surprising implications (see \S \ref{OTsection} for the precise statement) 

\begin{Proposition}
\begin{itemize}
    \item [(A)] There exist overtwisted contact $3$-manifolds that have an exotic loop of Lutz Twist embeddings.
    \item [(B)] There exist overtwisted contact $3$-manifolds that have an exotic loop of standard sphere embeddings.
\end{itemize}
\end{Proposition}

In other words, (A) says that the \textit{h}-principle for codimension $0$ isocontact embeddings of embedded $S^1$-families of overtwisted disks fails in $1$-parametric families,  see \cite{GromovPDR,EliashbergMishachevHPrinciple}. To the best of our knowledge this is the first example of this nature. On the other hand, (B) says that the \textit{h}-principle for standard spheres \cite{fmp} in tight contact $3$-manifolds fails in the overtwisted case. 

The first known exotic phenomena regarding overtwisted disks in  overtwisted contact $3$-manifolds are due to Vogel \cite{Vogel}. He has proved that the space of overtwisted disks in certain overtwisted $3$-sphere is disconnected and used this to construct an exotic loop of overtwisted contact structures. By Eliashberg's $h$-princple \cite{EliashbergOT}, understanding the homotopy type of the space of overtwisted disks is the only obstacle remaining in order to completely understand the homotopy type of the space of overtwisted contact structures on a $3$-manifold. Thus, understanding families of overtwisted disks or overtwisted objects bears special importance in $3$-dimensional contact topology.


\subsection{Context} 

\subsubsection{h-principles} As with symplectic topology, an ubiquitous theme of contact topology is the contrast between two types of behaviours: flexible (similar to differential topology) and rigid (similar to algebraic geometry). Beyond the tight-overtwisted dichotomy, $3$-dimensional contact topology would seem to be dominated by \textit{flexibility}, due to the following $h$-principle of Eliashberg and Mishachev:
\begin{Theorem*}[\cite{EM}]\label{thm:EM}
Let $(\mathbb{B}^3 , \xi_{\mathrm{st}} = \mathrm{Ker}(dz -ydx ) )$ be the standard contact unit $3$-ball. Then the inclusion $\mathrm{Cont}(\mathbb{B}^3 , \xi_{\mathrm{st}} ) \rightarrow \mathrm{Diff}(\mathbb{B}^3 ) $ is a homotopy equivalence. 
\end{Theorem*}

Here $\mathrm{Cont}(\mathbb{B}^3,\xi )$ is the group of contactomorphisms of $Y$ fixing a neighbourhood of $\partial \mathbb{B}^3$, and likewise for the group of diffeomorphisms $\mathrm{Diff}(\mathbb{B}^3 )$. This result was claimed, without a complete proof, by Eliashberg in \cite{Eliashberg20}, where the 0-1 parametric case was treated. The complete proof recently appeared in \cite{EM}. To give some context, the analogous statement that $\mathrm{Diff}(\mathbb{B}^3 ) \rightarrow \mathrm{Homeo}(\mathbb{B}^3 )$ is a homotopy equivalence is equivalent to the Smale conjecture in dimension $3$, a deep result proved by Hatcher \cite{hatchersmale}. Then an argument due to Cerf \cite{cerf} shows that the Smale conjecture implies that $\mathrm{Diff}(Y) \rightarrow \mathrm{Homeo}(Y)$ is a homotopy equivalence for all $3$-manifolds. Thus, the exotic phenomena at the $\pi_0$-level which are exhibited in Theorems \ref{mainthm}-\ref{mainthm2} are in sharp contrast with the above and unexpected.

\begin{Remark}
We also note in passing that in four-dimensional symplectic topology the statement analogous to the $h$-principle of Eliashberg and Mishachev is false: for the standard symplectic $( \mathbb{R}^4 , \omega = dx\wedge dy + dz \wedge dw )$ the inclusion $$\mathrm{Symp}_{c}(\mathbb{R}^4  , \omega) \rightarrow \mathrm{Diff}_{c }(\mathbb{R}^4 )$$ is not a homotopy equivalence. This follows from Gromov's result on the contractibility of $\mathrm{Symp}_c (\mathbb{R}^4 , \omega )$ \cite{gromov} combined with Watanabe's recent disproof of the $4$-dimensional Smale Conjecture \cite{watanabe}.
\end{Remark}

\subsubsection{Gompf's contact Dehn twist}

We will see (\S \ref{dehnsection}) that the contact Dehn twist is well-defined on a (co-oriented) sphere $S \subset (Y,\xi )$ with a \textit{tight neighbourhood}. To the authors' knowledge, this contactomorphism was first considered by Gompf on the non-trivial sphere in the tight $S^1 \times S^2$, see \cite{gompf}. Gompf observed that $\tau_{S}$ and its iterates are not contact isotopic to the identity. Ding and Geiges \cite{DG10} later established that $\tau_{S}^2$ generates all smoothly trivial contact mapping classes (see also \cite{minlens}). Gironella \cite{Gironella} has recently studied higher dimensional analogues of Gompf's contactomorphism. However, all iterates of Gompf's $\tau_S$ and Gironella's generalisations happen to be \textit{formally non-trivial} already, and hence \textit{not} exotic. 

\subsubsection{Finite order exotic contactomorphisms}

The previously known exotic three-dimensional contactomorphisms have \textit{finite} order and the underlying $3$-manifolds have $b_1 \geq 3$. These were detected on torus bundles by Geiges and Gonzalo \cite{GG}, who used an essentially elementary argument to reduce the problem to the Giroux--Kanda classification of tight contact structures on $T^3$. This was reproved using contact homology by Bourgeois \cite{bourgeois}, who also found more exotic contactomorphisms in Legendrian circle bundles over surfaces of positive genus. In the latter case, those contactomorphisms have been shown to generate the group (\ref{Cont0}) by Geiges, Klukas \cite{GK14}, Giroux and Massot \cite{GM17}. Unlike the squared Dehn twists, these exotic contactomorphisms are all given by global symmetries. The paradigmatic example is the following:

\begin{Example}[\cite{GG,bourgeois}]
Consider the $3$-torus $T^3$ with the fillable contact structure $\xi_1 = \mathrm{Ker} \big( \cos \theta dx - \sin \theta dy \big)$. By passing to $n$-fold covers $T^3 \rightarrow T^3$, $(\theta, x, y ) \mapsto (n \theta , x, y )$ we obtain contact structures $\xi_n$ on $T^3$. By a classical result of Giroux and Kanda \cite{girouxT3,kandaT3} the contact structures $\xi_n$ ($n \geq 1$) are pairwise not contactomorphic and give all the tight contact structures on $T^3$. When $n \geq 2$ the deck transformations of the $n$-fold cover $T^3 \rightarrow T^3$ generate all the exotic contactomorphisms of $(T^3, \xi_n )$.  
\end{Example}

\subsubsection{Other Exotic Dehn twists}

Dehn twists have been a common source of exotic phenomena in topology:
\begin{enumerate}[label=(\alph*)]
    \item Let $Y_\# = Y_- \# Y_+$ be the sum of two aspherical $3$-manifolds $Y_\pm$. By a result of McCullough \cite{exotic3diff} (see also \cite{hatcherwahl}) it follows that the kernel of $\pi_0 \mathrm{Diff}(Y_\# ) \rightarrow \mathrm{Out} ( \pi_1 Y_\# )$ is $\cong \mathbb{Z}_2$, generated by the smooth Dehn twist on the separating sphere.

    \item Seidel \cite{seidellag} used Lagrangian Floer homology to detect exotic four-dimensional symplectomorphisms with infinite order in the symplectic mapping class group, given by squared Dehn twists on Lagrangian spheres. He later generalised these results to higher dimensions \cite{seidelgraded,seidelsequence}. See also the recent work of Smirnov \cite{smirnov1,smirnov2} using Seiberg-Witten gauge theory. 
    
    \item Kronheimer and Mrowka \cite{sumK3} have proved that the smooth Dehn twist on the separating sphere in the connected sum of two copies of the smooth $4$-manifold underlying a $K3$ surface is not smoothly isotopic to the identity, even if it is topologically. For this they employ the Bauer-Furuta homotopical refinement of the Seiberg-Witten invariants of $4$-manifolds. See also \cite{linK3}.

\end{enumerate}

\subsection{Sketch of the proof of Theorem \ref{mainthm}(A)}



We outline here a proof of Theorem \ref{mainthm}(A) which is simpler than the one we give in detail in the article. In particular, the proof that we present now does not yield the stronger conclusion that the class of $\tau_{S_\#}^2$ is non-trivial in the abelianisation of (\ref{Cont0}). We will need a stronger argument which uses Theorem \ref{sumcontact2} in order to deduce both this and Theorem \ref{mainthm2}.

The main ideas go as follows. First, we have a \textit{relative} version of the problem. Given a Darboux ball $B$ in a contact $3$-manifold $(Y, \xi )$ we have a contactomorphism given by a Dehn twist $\tau_{\partial B}$ performed on an exterior sphere parallel to $\partial B$. This contactomorphism fixes the ball $B$ and need not be contact isotopic to the identity \textit{relative} to $B$, even if it always is globally (not fixing the ball). The problem of whether the squared Dehn twist $\tau_{\partial B}^2$ is isotopic rel. $B$ to the identity can be essentially recast as a lifting problem involving families of contact structures: if $Y$ is aspherical (i.e. irreducible and with infinite fundamental group) then $\tau_{\partial B}^2$ is isotopic to the identity rel. $B$ precisely when the fibration given by the evaluation map $ev : \mathcal{C}(Y , \xi ) \rightarrow S^2$ admits a (homotopy) section (see Corollary \ref{criterionB}). We recall that $ev$ is defined by evaluating contact structures at a point. The key point that we exploit is that this fibration resembles a corresponding "evaluation map" pertaining the Seiberg-Witten gauge theory of the manifold $Y$, and which is closely related to the $U$ map in monopole Floer homology. As a result, an obstruction to the existence of a section was given by the second author in \cite{yo}: if $\mathbf{c}(\xi ) \notin \mathrm{Im}U$ then no (homotopy) section exists, and thus $\tau_{\partial B}^2$ isn't isotopic to the identity rel. $B$. 

Going back to the original problem, consider two \textit{tight} irreducible contact manifolds $(Y_\pm , \xi_\pm )$ and their sum $(Y_\#, \xi_\# )$. Let $ \mathrm{CEmb} \big( S^2 , (Y_\# , \xi_\# ) \big)_{S_\#}$ be the space of co-oriented \textit{convex} embeddings $S^2 \hookrightarrow (Y_\#, \xi_\# )$ with standard characteristic foliation, in the isotopy class of the separating sphere $S_\#$.
The group of contactomorphisms of $(Y_\#, \xi_\# )$ acts transitively on this space and yields a fibration\footnote{Strictly speaking, we should replace $\mathrm{Cont}(Y_\# , \xi_\# )$ with the subgroup consisting of contactomorphisms which preserve the isotopy class of the co-oriented sphere $S_\#$.} 
\begin{align}
    \mathrm{Cont}(Y_\#, \xi_\#, S_\# ) \rightarrow \mathrm{Cont}(Y_\#, \xi_\# ) & \rightarrow \mathrm{CEmb} \big( S^2 , (Y_\# , \xi_\# ) \big)_{S_\#}. \label{fib1} \\
    f & \mapsto f (S_\# ) \nonumber
\end{align}
From the long exact sequence of homotopy groups, a contactomorphism $f$ of $(Y_\#, \xi_\#)$ fixing the sphere $S_\#$ is contact isotopic to the identity (not necessarily fixing $S_\#$) precisely when it arises as the monodromy in (\ref{fib1}) of a loop of sphere embeddings. It thus becomes essential to understand the topology of the sphere embedding space. This brings us to the following $h$-principle type result, which asserts that the topological complexity of this space only comes from reparametrisations of the source:

\begin{Theorem}\label{embthm}
If $(Y_\pm , \xi_\pm )$ are irreducible and tight then the reparametrisation map provides a homotopy equivalence $\mathrm{U}(1) \xrightarrow{\simeq} \mathrm{CEmb}\big( S^2 , (Y_\# , \xi_\# ) \big)_{S_\#}$.
\end{Theorem}

In the smooth case, the result analogous to the above was proved by Hatcher \cite{S1xS2}. The proof of Theorem \ref{embthm} rests on $h$-principle for standard convex spheres established by the first author together with J. Martínez-Aguinaga and F. Presas \cite{fmp}, and should be regarded as an application of the $h$-principle of Eliashberg and Mishachev \cite{EM}.

With these ingredients in place, the proof of Theorem \ref{mainthm}(A) goes as follows. The monodromy in (\ref{fib1}) over the standard loop in $\mathrm{U}(1)$ is given by the product of Dehn twists $\tau_{\partial B_-} \tau_{\partial B_+}$ (see Lemma \ref{reparamLemma}). The contact Dehn twist $\tau_{S_\#}$ agrees with the image of $\tau_{\partial B_-}$ in $\pi_0 \mathrm{Cont}(Y_\#, \xi_\# )$. Because the manifolds $(Y_\pm , \xi_\pm )$ have infinite order contact Dehn twists $\tau_{\partial B_\pm}$ rel. $B_\pm$, then for all $k \geq 1$ the class $\tau_{\partial B_-}^k \in \pi_0 \mathrm{Cont}(Y_\#, \xi_\#, S_\# )$ is not an iterate of $\tau_{\partial B_-} \tau_{\partial B_+}$ or its inverse. It follows that $\tau_{S_\#}$ and its iterates are not contact isotopic to the identity in $(Y_\#, \xi_\# )$.

\subsection{Outline}

The structure of the article is as follows. In \S \ref{backgroundsection} we introduce notation and present background material. In \S \ref{dehnsection} we define the contact Dehn twist, establish various key properties and present examples where it is isotopic to the identity. In \S \ref{monopolesection} we provide background on the families version of the Kronheimer--Mrowka contact invariant introduced in \cite{yo}, which will be one of the main ingredients in the proofs of our main results. In \S \ref{spheressection} we review the $h$-principle for families of convex spheres in tight contact $3$-manifolds established in \cite{fmp}. In \S \ref{connectedsumsection} we use this $h$-principle to establish Theorem \ref{sumcontact2}. We then complete the proofs of Theorems \ref{mainthm} and \ref{mainthm2}. In \S \ref{OTsection} we deduce exotic $1$-parametric phenomena in overtwisted contact $3$-manifolds.

\subsubsection*{Acknowledgements} The first named author would like to acknowledge his advisor Francisco Presas for valuable conversations. The second named author thanks his advisor Francesco Lin for his support and encouragement, together with Hyunki Min for useful coversations. We would also like to thank the anonymous referee for the careful suggestions from which this manuscript has greatly benefited. The second author was partially supported by NSF grant DMS-2203498.

\section{Background}

This section introduces the main players in this article: spaces of contact structures, contactomorphisms, embeddings, etc.

\begin{Remark} For convenience, throughout this article by a "fibration" we will mean a "Serre fibration". By a "homotopy equivalence" we will mean a "weak homotopy equivalence". However, the latter distinction isn't important: the various infinite dimensional spaces that we consider are Fréchet manifolds, hence they have the homotopy type of countable CW complexes \cite{palaisinfinite,milnorcw} and Whitehead's Theorem applies.
\end{Remark}

\label{backgroundsection}

\subsection{Notation}
Let $(Y , \xi )$ be a closed contact $3$-manifold. We always assume $Y$ is connected and oriented, and $\xi$ co-oriented and positive. Occasionally we will allow $Y$ to be compact with non-empty boundary, in which case we assume that $\partial Y$ is \textit{convex} for the contact structure $\xi$ and we fix a collar neighbourhood $C = (-1,0] \times \partial Y$ of $\partial Y$. We quickly introduce here some of the spaces that will be relevant in the article, all of which are equipped with the Whitney $C^\infty$ topology:
\begin{itemize}

\item We denote by $\mathrm{Emb}\big( \mathbb{B}^3 , Y \big) $ the space of orientation-preserving smooth embeddings $\phi : \mathbb{B}^3 \hookrightarrow Y$ of the closed unit ball (avoiding the closure of $C$, if $\partial Y \neq \emptyset$). Let $\mathrm{Emb}\big( (\mathbb{B}^3 , \xi_{\mathrm{st}} ) , (Y, \xi ) \big)$ be the subspace consisting of contact embeddings of the standard contact unit ball. Such embeddings will be referred to as \textit{Darboux balls} in $(Y , \xi )$. Darboux's theorem asserts that for any interior point $p$ of a contact manifold we may find such $\phi$ with $\phi (0 ) = p$. We will often incur in abuse of notation by referring to a Darboux ball only by its image $B := \phi ( \mathbb{B}^3 )$.

\item We denote by $\mathrm{Diff}(Y)$ the group of orientation-preserving diffeomorphisms, and by $\mathrm{Diff} (Y , B)$ the subgroup consisting of those which fix a Darboux ball $B$ pointwise. By $\mathrm{Diff}_0 (Y)$ and $\mathrm{Diff}_0 (Y , B)$ we denote the subgroups consisting of those which are smoothly isotopic to the identity (rel. $B$ in the second case). We denote by $\mathrm{Cont}(Y) \subset \mathrm{Diff}$ the subgroup of co-orientation preserving contactomorphisms of $(Y ,\xi)$, and by $\mathrm{Cont} (Y , B)$ the subgroup consisting of those which fix a Darboux ball $B$ pointwise. By $\mathrm{Cont}_0 (Y)$ and $\mathrm{Cont}_0 (Y , B)$ we denote the subgroups consisting of those which are \textit{smoothly} isotopic to the identity (rel. $B$ in the second case). 

\item We denote by $\mathcal{C}(Y , \xi )$ the space of contact structures on $Y$ in the path-component of $\xi$. When $\partial Y \neq \emptyset$ then we also require that they agree with $\xi$ over $C$. Given a Darboux ball $B $ in $(Y , \xi )$ we denote by $\mathcal{C}(Y , \xi , B )$ the subspace consisting of contact structures $\xi^\prime$ for which the coordinate ball $B$ is a Darboux ball for $(Y , \xi^{\prime} )$ (i.e. $\xi = \xi^\prime$ over $B$).

\item We denote by $\mathrm{Fr}(Y)$ the principal $(\mathrm{SO}(3) \simeq ) \mathrm{GL}^{+}(3)$-bundle over $Y$ of oriented frames in $TY$, and by $\mathrm{CFr}(Y)$ the principal $(\mathrm{U}(1) \simeq) \mathrm{CSp}^+(2, \mathbb{R})$-bundle over $Y$ of co-oriented frames in $\xi$. Here, $\mathrm{CSp}^+(2, \mathbb{R})$ denotes the linear conformal-symplectomorphism group. By the smooth and contact versions of the Disk Theorem\footnote{The key point in the contact case is that $\varphi_t(x,y,z) :=(tx,ty,t^2z)$ is a contactomorphism of $(\mathbb{R}^3,\xi_{\mathrm{st}})$ for every $t>0$, so the proof in the contact case follows along the same lines as in the smooth case (see \cite{geiges}, Theorem 2.6.7).} we have homotopy equivalences 
\begin{align}
     \mathrm{Emb}(\mathbb{B}^3 , Y ) & \xrightarrow{\simeq} \mathrm{Fr}(Y)  \label{embfr}\\
     \phi & \mapsto (d\phi )_0 (e_1 , e_2 , e_3 ) \nonumber\\
     \mathrm{Emb}( (\mathbb{B}^3 , \xi_{\mathrm{st}}) , (Y , \xi ) )& \xrightarrow{\simeq} \mathrm{CFr}(Y, \xi )\nonumber\\
     \phi & \mapsto (d \phi )_0 (e_1 , e_2 ).\nonumber
\end{align}
Notice that $\mathrm{Fr}(Y)\simeq Y \times \mathrm{SO}(3)$ and, when the Euler class of $\xi$ vanishes, $\mathrm{CFr}(Y, \xi ) \simeq Y \times \mathrm{U}(1)$.
\item We denote by $\mathrm{Emb}(S^2 , Y )$ the space of co-oriented embeddings of $2$-spheres. By $\mathrm{CEmb}\big( S^2 , (Y, \xi ) \big)$ we denote the subspace consisting of \textit{convex} embeddings with \textit{standard characteristic foliation} ("standard convex spheres" in short). Recall that a surface $\Sigma \subset (Y, \xi )$ is convex \cite{convexite}\cite{geiges} if there exists a contact vector field on a neighbourhood which is transverse to $\Sigma$. The standard characteristic foliation on $S^2$ is that induced from its embedding as the boundary of the Darboux ball. 

\item We denote by $\mathrm{Cont}(Y, \xi , S)$ the subgroup of contactomorphisms which fix a standard convex sphere $S$ pointwise, and likewise for $\mathrm{Diff}(Y, S )$.
\end{itemize}

\subsection{Standard fibrations}\label{fibrations}

Next, we review how the spaces introduced above relate to each other through various natural fibrations. Some of the material from this section is treated in \cite{GM17} in greater detail.

\subsubsection{Diffeomorphisms acting on contact structures}
By an application of Gray's stability Theorem (a.k.a Moser's argument) \cite{geiges} with parameters one can show
\begin{Lemma}\label{lem:Gray} The action $f \mapsto f_\ast \xi$ of the group of diffeomorphisms on a fixed contact structure $\xi$ gives a fibration
\begin{align}
\mathrm{Cont}_0 (Y, \xi ) \rightarrow \mathrm{Diff}_0 (Y ) \rightarrow \mathcal{C}(Y , \xi ).\label{moser}
\end{align}
Similarly, there is fibration 
\begin{align}
\mathrm{Cont}_0 (Y, \xi , B ) \rightarrow \mathrm{Diff}_0 (Y , B ) \rightarrow \mathcal{C}(Y , \xi , B ). \label{moserB}
\end{align}
\end{Lemma}

By (\ref{moser}), understanding the homotopy type of the space of contact structures $\mathcal{C}(Y , \xi )$ and the group of contactomorphisms $\mathrm{Cont}_0 (Y , \xi )$ is essentially equivalent, since the homotopy type of $\mathrm{Diff}_0 (Y )$ is often well-understood (e.g. for all prime $3$-manifolds by now).

\subsubsection{Contactomorphisms acting on Darboux balls}
By an application of the contact isotopy extension Theorem \cite{geiges} with parameters we have
\begin{Lemma} The action $f \mapsto f(B)$ of the group of contactomorphisms on a fixed Darboux ball $B \subset Y$ gives a fibration
\begin{align}
\mathrm{Cont} ( Y , \xi , B ) \rightarrow \mathrm{Cont} (Y, \xi ) \rightarrow \mathrm{Emb}( (\mathbb{B}^3 , \xi_{\mathrm{st}}) , (Y , \xi ) ). \label{contball}
\end{align}
Similarly, there is a fibration
\begin{align}
 \mathrm{Diff} ( Y  , B ) \rightarrow \mathrm{Diff} (Y ) \rightarrow \mathrm{Emb}( \mathbb{B}^3 ,  Y  ). \label{diffball}
 \end{align}
\end{Lemma}

\subsubsection{Evaluation of contact structures at a point}

Fix a Darboux ball $B \subset Y$ with center $0 \in Y$. By regarding the $2$-sphere $S^2$ as the space of co-oriented planes in the tangent space $T_0 B$ we obtain the \textit{evaluation map}
\begin{align}
ev_B : \mathcal{C}(Y , \xi )  \rightarrow S^2 \quad , \quad
\xi^{\prime}  \mapsto \xi^{\prime} (0). \label{ev}
\end{align}

The following result is well-known but we provide a proof:

\begin{Lemma}\label{evfib}
The evaluation map (\ref{ev}) is a fibration. The inclusion $\mathcal{C}(Y , \xi , B ) \rightarrow (ev_B )^{-1} (\xi (0 ) ) $ is a homotopy equivalence.
\end{Lemma}
\begin{proof}
Let $\mathbb{B}^j$ be the unit $j$-disk and consider a homotopy $[0,1] \times \mathbb{B}^j \rightarrow S^2$, $(t,u) \mapsto \sigma_{t , u}$, together with a lift of the time zero map $\{0\} \times \mathbb{B}^j \rightarrow \mathcal{C}(Y , \xi )$, $u \mapsto \xi_u$ i.e. at the point $0 \in B$ we have $\xi_u (0 ) = \sigma_{0,u}$. We must find a family of contact structures $\xi_{t,u}$ with $\xi_{t,u}(0 ) = \sigma_{t,u}$ and $\xi_{0,u} = \xi_u$.

Let $v_{t,u} \in S(T_0 B ) = S^2$ be the unit normal (with respect to the standard flat metric on $B$) to the plane $\sigma_{t,u}$. Since the action of $\mathrm{SO}(3)$ on $S^2$ gives a fibration $\mathrm{SO}(3) \rightarrow S^2$, $A \mapsto A e_3 $, then we may find $A_{t,u} \in \mathrm{SO}(3)$ such that $A_{t , u}e_3 = v_{t,u}$. Differentiating $A_{t,u}$ in $t$ we get a vector field on $V_{t,u}$ on $\mathbb{R}^3$. After cutting off $V_{t,u}$ outside the unit ball $B \subset Y$ we regard $V_{t,u}$ as an $u$-family of $t$-dependent vector fields on $Y$ whose associated flows (starting at time $t = 0$) we denote $\phi^{t}_{u}$. We obtain contact structures $\xi_{t,u} := (\phi^{t}_{u})_\ast \xi_u$ with the desired property, which in fact agree with $\xi$ outside $B \subset Y$.

For the second part, let $\xi_u = \mathrm{Ker} \alpha_u$ be a family of contact structures parametrised by a disk $ \mathbb{B}^j \ni u$ so that $\xi_u (0) = \xi (0) $ for all $u\in \mathbb{B}^j$ and $\xi_u(p)=\xi(p)$ for all $(u,p)\in \partial \mathbb{B}^j\times B$. We must deform relative to $\partial \mathbb{B}^j$ this family of contact structures to another family which agrees with $\xi$ over the Darboux ball $B$. Denote by $i:\mathbb{B}^3\hookrightarrow Y$ the inclusion of $B=i(\mathbb{B}^3)$. By the parametric version of Darboux's Theorem 
we obtain a family of disk embeddings $\phi_u : \mathbb{B}^3 \hookrightarrow Y$, which are Darboux balls for $\xi_u$, such that $\phi_u (0) = 0 \in B$ and $(d\phi_u )_0 = \mathrm{id}$ for all $u\in \mathbb{B}^j$; and $\phi_u=i$  for all $u\in \partial \mathbb{B}^j$. By (\ref{embfr}) we may deform the family of embeddings $\phi_u$ to the inclusion $i$ relative to $\partial \mathbb{B}^j$, and this deformation may be followed by an isotopy $f_{u,t}\in \mathrm{Diff}(Y)$, $(u,t)\in \mathbb{B}^j\times [0,1]$, with $f_{u,t}=\mathrm{id}$ for all $(u,t)\in \mathbb{B}^j\times \{0\}\cup\partial \mathbb{B}^j\times [0,1]$. The homotopy of contact structures $(f_{u,t})_\ast \xi_u$, $(u,t)\in \mathbb{B}^j\times [0,1]$, solves the problem since $(f_{u,1})_\ast \xi_u$ agree with $\xi$ over $B$ for all $u\in \mathbb{B}^j$.
\end{proof}


\subsubsection{Contactomorphisms act on standard convex spheres}

Again, an application of the contact isotopy extension Theorem gives 

\begin{Lemma}
The action $f \mapsto f(S)$ of the group of contactomorphisms on a fixed standard convex sphere $S \subset Y$ gives a fibration
\begin{align}
    \mathrm{Cont}(Y, \xi , S ) \rightarrow \mathrm{Cont}(Y, \xi ) \rightarrow \mathrm{CEmb}\big( S^2 , (Y, \xi ) \big) \label{contsphere}
\end{align}
Similarly, there is a fibration
\begin{align}
 \mathrm{Diff} ( Y  , S ) \rightarrow \mathrm{Diff} (Y ) \rightarrow \mathrm{Emb}( S^2,  Y  ). \label{diffsphere}
 \end{align}
\end{Lemma}

\begin{Remark}
The above statement isn't quite precise. For either (\ref{contsphere}) or (\ref{diffsphere}), the downstairs projection is not surjective in general, so strictly speaking we only have a fibration over a union of connected components of the right-hand side. We will make no further comment on this point from now on.
\end{Remark}

\subsection{Formal triviality and exoticness} \label{exotic}

Here we collect basic material that we need related to the notion of a formal contactomorphism. The material in this section should be well-known to experts but we did not find a convenient reference.
 
\subsubsection{Formal contact structures and contactomorphisms}

For a $3$-manifold $Y$, the flexible analogue\footnote{In general, if $Y$ has dimension $2n+1 \geq 3$ one should define $\Xi (Y , \xi )$ as the space of codimension $1$ hyperplane fields in $TY$ \textit{equipped with} a $\mathrm{U}(n)$ structure.} 
of a contact structure is a \textit{$2$-plane field} i.e. a codimension $1$ distribution $\xi \subset TY$. All $2$-planes in a $3$-manifold are assumed to be co-oriented from now on, as we've been assuming with contact structures. Let $\Xi(Y , \xi )$ denote the path-component of a fixed $2$-plane field $\xi$ in the space of all such. If $\xi$ is a contact structure we have a natural inclusion map $\mathcal{C}(Y , \xi ) \rightarrow \Xi (Y , \xi )$. The correct flexible analogue of a contactomophism is:

\begin{Definition} A \textit{formal contactomorphism} of $(Y ,\xi )$ (where $\xi$ is a $2$-plane field) is a pair $(f , \{ \phi^s \}_{0 \leq s \leq 1} )$ such that $f \in \mathrm{Diff}(Y)$ and $\{ \phi^s \}_{0 \leq s \leq 1}$ is a homotopy through vector bundle isomorphisms $\phi^s : TY \xrightarrow{\cong} f^\ast TY$ such that $\phi^0 = df$ and $\phi^1$ preserves the $2$-plane field $\xi$. 
\end{Definition}

Of course, the above notion can be generalised to an arbitrary $n$-manifold equipped with a reduction of structure group to a subgroup $G \subset \mathrm{GL}(n , \mathbb{R} )$. The group of formal contactomorphisms of $(Y , \xi )$ is denoted $\mathrm{FCont}(Y, \xi )$. When $\xi$ is a contact structure there is the obvious inclusion map $\mathrm{Cont} (Y , \xi ) \rightarrow \mathrm{FCont}(Y , \xi )$ given by $f \mapsto (f , df )$ (where $df$ denotes the constant homotopy at $df$).

A homotopy class in $\pi_j \mathrm{Cont}(Y , \xi )$ is said to be \textit{formally trivial} if it lies in the kernel of $ \pi_j \mathrm{Cont}(Y , \xi ) \rightarrow \pi_j \mathrm{FCont} (Y , \xi )$. If, in addition, such a homotopy class is non-trivial in $\pi_j \mathrm{Cont}(Y , \xi )$ then we call it \textit{exotic}. Similar terminology applies for families of contact structures.


\subsubsection{A flexible analogue of (\ref{moser})}

We introduce a flexible counterpart of the fibration (\ref{moser}). This is done via fibrant replacement of the map $\mathrm{Diff}_0 (Y) \rightarrow \Xi (Y, \xi )$ , $f \mapsto f^\ast \xi$. That is, we decompose this map as the composite of a homotopy equivalence $\mathrm{Diff}_0 (Y) \xrightarrow{\simeq} \mathrm{FDiff}_0 (Y)$ and a fibration $\mathrm{FDiff}_0 (Y) \rightarrow \Xi (Y, \xi )$. Here $\mathrm{FDiff}(Y)$ is the topological group which consists of pairs $(f , \{ \phi^t \}_{0 \leq t \leq 1} )$ where $f \in \mathrm{Diff}(Y)$ and $\{ \phi^t \}_{0 \leq t \leq 1}$ is a homotopy of vector bundle isomorphisms $\phi^t : TY \xrightarrow{\cong} f^\ast TY$ such that $\phi^0 = df$. By $\mathrm{FDiff}_0 (Y )$ we denote the identity component. Clearly the inclusion induces a homotopy equivalence $\mathrm{Diff}(Y) \simeq \mathrm{FDiff}(Y)$. Define a mapping
\begin{align}
    \mathrm{FDiff}_0 (Y ) & \rightarrow \Xi (Y , \xi ) \label{moserf} \\
    (f , \{\phi^t\} ) & \mapsto \phi^1 (\xi ) \nonumber
\end{align}
\begin{Lemma}\label{moserflemma}
Let $\xi$ be a $2$-plane field on a compact oriented $3$-manifold $Y$. Then the mapping (\ref{moserf}) is a fibration with fiber $\mathrm{FCont}_0 (Y , \xi )$. Thus, for a contact structure $\xi$ we have a commuting diagram of fibrations inducing a homotopy equivalence of total spaces
\[
\begin{tikzcd}
\mathrm{FCont}_0 (Y , \xi )  \arrow{r} & \mathrm{FDiff}_0 (Y ) \arrow{r} &  \Xi (Y , \xi )\\
\mathrm{Cont}_0 (Y , \xi )  \arrow{u} \arrow{r} & \arrow{u}{\simeq} \mathrm{Diff}_0 (Y ) \arrow{r} &  \arrow{u} \mathcal{C} (Y , \xi )\\
\end{tikzcd}
\]
\end{Lemma}

We leave the proof of this Lemma as an exercise. It follows:

\begin{Corollary}\label{comp}
Let $(Y , \xi )$ be a contact $3$-manifold. If $\beta \in \pi_j \mathcal{C}(Y , \xi )$ is formally trivial, then so is its image in $\pi_{j-1} \mathrm{Cont}_0 (Y , \xi )$ under the connecting map of the fibration (\ref{moser}).
\end{Corollary}

The homotopy type of the space $\Xi (Y, \xi )$ is often easy to understand, unlike that of $\mathcal{C}(Y , \xi )$. 
\begin{Example}
Let $Y$ be any integral homology $3$-sphere, and $\xi$ a $2$-plane field on $Y$. Let $\xi_{\mathrm{st}}$ be any contact structure on $S^3$ (say, the tight one). By a result of Hansen \cite{hansen} there is a homotopy equivalence $\Xi (S^3, \xi_{\mathrm{st}} ) \simeq \Xi (Y , \xi )$. From this one easily calculates $$ \pi_j \Xi (Y , \xi ) \approx \pi_j S^2 \times \pi_{j+3} S^2. $$
\end{Example}

\section{Contact Dehn twists on spheres} \label{dehnsection}

In this section we define the contact Dehn twist on a sphere in several equivalent ways, establish some key properties and discuss some examples when its square is isotopic to the identity.

\subsection{The contact Dehn twist}\label{defn1}

Let $(Y,\xi)$ be a contact $3$-manifold, and $S \subset Y$ be a co-oriented embedded sphere. Provided $S$ has a tight neighbourhood, we can associate to $S$ a contactomorphism $\tau_S$ well-defined in $\pi_0 \mathrm{Cont}(Y, \xi )$. We discuss this construction now.

\subsubsection{Local model} We start by discussing the local picture. Consider the contact $3$-manifold $Y_0 = [-1 , 1 ] \times S^2$ with the tight contact structure $\xi_0 = \mathrm{Ker}(\alpha_0 )$ where $\alpha_0 =  z ds + \frac{1}{2}xdy - \frac{1}{2} y dx $. Here $s$ is the standard coordinate on $[-1 , 1 ]$ and $x, y , z$ coordinates on $\mathbb{R}^3$ restricted onto the unit sphere $S^2$. Consider the sphere $S_0 = \{0\} \times S^2 \subset Y_0$. We now describe the contact Dehn twist $\tau_{S_0}$ on the sphere $S_0$.

We choose a smooth function $\theta : [-1,1] \rightarrow [0, 2 \pi ]$ with $\theta(s) \equiv 0$ near $s = -1$ and $\theta (s) = 2 \pi  $ near $s = 1$. Let $R_\varphi $ be the counterclockwise rotation in the $xy$ plane with angle $\varphi$. Consider the diffeomorphism $\widetilde{\tau}_{S_0}$ of $Y_0$ given by a smooth Dehn twist along $S_0$
$$ \widetilde{\tau}_{S_0} (s , x,y,z) = (s , R_{\theta(s)} (x,y) , z ).$$

Since $\pi_1 \mathrm{SO}(3) = \mathbb{Z}/2$ it follows that the squared Dehn twist $\widetilde{\tau}_{S_0}^2$ is smoothly isotopic to the identity rel. $\partial Y_0$. We don't quite have a contactomorphism of $(Y_0 , \xi_0 )$ since 
$$ \widetilde{\tau}_{S_0}^\ast \alpha_0 = \alpha_0 + \frac{\theta^{\prime}(s)}{2} (x^2 + y^2 ) ds .$$ However, consider the naive interpolation from $\alpha_0$ to $\widetilde{\tau}_{S_0}^\ast \alpha_0$ 
\begin{align*}
    \alpha_t = \alpha_0 + t \frac{\theta^{\prime}(s)}{2}(x^2 + y^2 ) ds
\end{align*}
and observe that
\begin{Lemma}\label{interp}
For any $t \in [0,1]$ the form $\alpha_t$ is a contact form.
\end{Lemma}
\begin{proof}
A straightforward calculation shows $\alpha_t \wedge d \alpha_t = \alpha_0 \wedge d \alpha_0 > 0$. 
\end{proof}

Thus, by Gray stability (a.k.a Moser's argument) \cite{geiges} the deformation of contact structures $\xi_t = \mathrm{Ker} (\alpha_t )$ is realised by an isotopy $f_t$ i.e. $f_0 = \mathrm{id}$ and $(f_t)^\ast \xi_t = \xi_0$. Since the forms $\alpha_t$ don't depend on $t$ near $\partial Y_0$ we may further assume that $f_t = \mathrm{id}$ near $\partial Y_0$. We then replace $\widetilde{\tau}_{S_0}$ with $\tau_{S_0} :=  \widetilde{\tau}_{S_0} \circ f_1 $ and the latter is a contactomorphism of $(Y_0 , \xi_0 )$. We also have that that the support of $\tau_{S_0}$ can be made arbitrarily close to the sphere $S_0$ by choosing $\theta(s)$ appropriately. Then, for any $\epsilon \in (0,1]$ we have a well-defined isotopy class of contact Dehn twist $$ \tau_{S_0} \in \pi_0 \mathrm{Cont}([-\epsilon , \epsilon ] \times S^2 , \xi_0 ).$$ 

It is worth pointing out the following

\begin{Lemma}\label{lem:ContSphericalAnnulus} The group $\mathrm{Cont}(Y_0,\xi_0)$ is homotopy equivalent to $\Omega \mathrm{U}(1) \simeq \mathbb{Z}$. Its $\pi_0$ is generated by the contact Dehn twist $\tau_{S_0}$.
\end{Lemma}
\begin{proof}
Gluing a Darboux ball $B$ to $(Y_0,\xi_0)$ gives back the standard contact ball $(\mathbb{B}^3,\xi_\mathrm{st})$. Thus, from the fibration (\ref{contball}) we have a map of fiber sequences 
\begin{center}
\begin{tikzcd}
    \mathrm{Cont}(Y_0, \xi_0 ) \arrow{r} &  \mathrm{Cont}(\mathbb{B}^3,\xi_\mathrm{st})\arrow{r} &  \mathrm{Emb}\big(  (\mathbb{B}^3,\xi_\mathrm{st}) , (\mathbb{B}^3,\xi_\mathrm{st})\big) \\
    \Omega \mathrm{U}(1) \arrow{r} \arrow{u} & \{\ast\} \arrow{r} \arrow{u}{\simeq} & \mathrm{U}(1) \arrow{u}{\simeq}
\end{tikzcd}
\end{center}
where the middle homotopy equivalence follows from Theorem \ref{thm:EM} combined with Hatcher's Theorem \cite{hatchersmale}. The first assertion now follows. For the second assertion, we need to show that the generator $1 \in \pi_1 \mathrm{U}(1)$ maps to the class of the contact Dehn twist $\tau_{S_0}$ under the connecting map.

We first describe the contact Dehn twist on $S_0$ more conveniently in terms of the coordinates on the ball $\mathbb{B}^3 = B \cup Y_0$. Recall that the standard contact structure on $\mathbb{B}^3$ is $\xi_{\mathrm{st}} = \mathrm{Ker}\alpha_{\mathrm{st}}$ where $\alpha_{\mathrm{st}} = dz + \frac{1}{2}xdy - \frac{1}{2}ydx$. Choose a smooth function $\theta : [0 , 1] \rightarrow [0, 2\pi]$ with $\theta = 0 $ near $0$ and $\theta = 2 \pi$ near $1$. Let $r^2 := x^2 + y^2 + z^2$ be the radius squared function on $\mathbb{B}^3$. Then the diffeomorphism of $\mathbb{B}^3$ given by 
\begin{align*}
    \widetilde{\tau} (x,y,z) := (R_{\theta (r^2 )}(x,y), z )
\end{align*}
does not quite preserve the contact structure, but
\begin{align*}
    (\widetilde{\tau})^\ast \alpha_{\mathrm{st}} = \alpha_{\mathrm{st}} + \frac{1}{2}(x^2 + y^2 ) \theta^{\prime} (r^2 ) d(r^2 ).
\end{align*}
As in Lemma \ref{interp}, the obvious interpolation that takes the second term in the above identity to zero gives a path of \textit{contact} forms, and as in \S \ref{generaldehn} we may canonically deform $\widetilde{\tau}$ to a contactomorphism $\tau_{S_0}$ in the isotopy class of the contact Dehn twist on $S_0$.

Consider now a homotopy of maps $\theta_t : [0,1] \rightarrow [0,2\pi]$ with $\theta_t$ constant near $1$ (with value $2\pi$), such that $\theta_0 = \theta$ and $\theta_1$ is the constant function with value $2 \pi$. We obtain an isotopy through diffeomorphisms of $\mathbb{B}^3$ (fixing a neighbourhood of the boundary $\partial \mathbb{B}^3$, but not the smaller ball $B$!) given by 
\begin{align*}
    \widetilde{\tau}_t (x,y,z) := (R_{\theta_t (r^2 ) }(x,y) , z )
\end{align*}
such that $\widetilde{\tau}_0 = \widetilde{\tau}$ and $\widetilde{\tau}_1= \mathrm{id}$. Again, by observing that for each $t$ the obvious interpolation from $(\widetilde{\tau}_t )^\ast \alpha_{\mathrm{st}}$ and $\alpha_{\mathrm{st}}$ gives a path of contact forms, we may canonically deform the isotopy $\widetilde{\tau}_t$ to a \textit{contact} isotopy $\tau_{t}$ with $\tau_{0} = \tau_{S_0}$ and $\tau_1 = \mathrm{id}$.

Now, the path of contactomorphisms $\tau_{1-t}$ from the identity to $\tau_{\partial B}$ induces a \textit{loop} of Darboux balls $(\tau_{1-t})(B)$ in the class of the generator $1 \in \mathbb{Z} = \pi_1 \mathrm{Emb} \big( (\mathbb{B}^3, \xi_{\mathrm{st}} ) , (\mathbb{B}^3, \xi_{\mathrm{st}} ) \big)$. From this the required result now follows.
\end{proof}

Likewise, we have a firm hold on the topology of the space of standard spheres in our local model. Let $S_{\pm} = \{\pm 1/2\} \times S^2 \subset Y_0$, and denote by $e_0 : S^2 \hookrightarrow Y_0$ the embedding of $S_0 \subset Y_0$.

\begin{Lemma}\label{lem:ReparametrizationLoop}
The map induced by reparametrisation of $e_0$ $$ \mathrm{U}(1)\rightarrow \mathrm{CEmb}(S^2,(Y_0,\xi_0)) \quad , \quad \theta \mapsto e_0 \circ r_\theta$$is a homotopy equivalence. Here $r_\theta (x ,y , z ) = (R_{\theta}(x,y) ,z  )$. Under the connecting homomorphism of the fibration (\ref{contsphere}) the generator of $\pi_1 \mathrm{U}(1) = \mathbb{Z}$ maps to the class $$ (\tau_{S_{-}})^{-1}\tau_{S_{+}} \in \pi_0 \mathrm{Cont}(Y_0 , \xi_0 , S_0 ).$$
\end{Lemma}
\begin{proof}
We have the following map of fiber sequences, with homotopy equivalences on the fiber and total space by Lemma \ref{lem:ContSphericalAnnulus}
\[ \begin{tikzcd}
\mathrm{Cont}(Y_0 , \xi_0 , S_0 ) \arrow{r} & \mathrm{Cont}(Y_0,\xi_0)\arrow{r} &  \mathrm{CEmb}(S^2,(Y_0,\xi_0))\\  
\Omega \mathrm{U}(1) \times \Omega\mathrm{U}(1) \arrow{u}{\simeq} \arrow{r} &  \Omega\mathrm{U}(1)  \arrow{r} \arrow{u}{\simeq}  & \mathrm{U}(1) \arrow{u} 
\end{tikzcd}
\]
This establishes both assertions.
\end{proof}

\subsubsection{General case}\label{generaldehn}

The robustness of our local picture allows us to consider contact Dehn twists in more general settings. We fix a $3$-manifold $(Y, \xi )$ together with a co-oriented \textit{standard convex sphere} $S \subset Y$ i.e. an embedded sphere whose characteristic foliation agrees with that of $S_0 \subset Y_0$ in the local model. It follows that neighbourhoods of $S \subset Y$ and $S_0 \subset Y_0$ are contactomorphic in a (homotopically) canonical fashion \cite{convexite,geiges}, and by making the support of $\tau_{S_0}$ sufficiently close to $S_0 $ we may therefore implant $\tau_{S_0}$ into $(Y, \xi )$ as a compactly supported contactomorphism $\tau_S$, which we refer to as the \textit{contact Dehn twist} on the co-oriented standard convex sphere $S \subset Y$. The class of $\tau_S$ in $\pi_0 \mathrm{Cont}(Y , \xi )$ only depends on the isotopy class of $S$ in the space of co-oriented standard convex spheres, defining a map of sets
\begin{align*}
    \pi_0 \mathrm{CEmb}(S^2 , (Y, \xi ) ) \rightarrow \pi_0 \mathrm{Cont}(Y, \xi ) \quad , \quad S \mapsto \tau_{S} 
\end{align*}

The contactomorphism $\tau_S$ makes sense more generally whenever $S \subset Y$ is a just a convex co-oriented sphere with a \textit{tight neighbourhood} $U$ (but not necessarily having standard characteristic foliation). Indeed, by Giroux's Criterion \cite{GirouxCriterion} the dividing set of $S$ is connected. Then by Giroux's Realisation theorem, we may find a smooth isotopy of sphere embeddings $S_t$ whose image lies in the tight neighbourhood $U$, $S_0 = S$ and $S_1$ is a \textit{standard} convex sphere, to which we associate the Dehn twist $\tau_{S_1}$ by the previous construction. A different choice of isotopy $S_{t}^{\prime}$ may yield a different standard convex sphere $S_{1}^{\prime} $. The two spheres ($S_{1}$ and $S_{1}^{\prime})$) are isotopic within $U$ as \textit{standard} convex spheres by a result of Colin (\cite{colin}, Proposition 10), so the contact Dehn twists $\tau_{S_1}$ and $\tau_{S_{1}^{\prime}}$ are contact isotopic. Therefore, we have a well defined contact Dehn twist $\tau_S \in \pi_0 \mathrm{Cont}(Y, \xi )$ associated to the convex sphere $S$ with tight neighbourhood $U$. In fact, since any \textit{smooth} sphere can be made convex by a small isotopy \cite{convexite}, this construction defines a map
\begin{align*}
    \pi_0 \mathrm{Emb}_{\mathrm{tight}}(S^2 , (Y, \xi ) )  \rightarrow \pi_0 \mathrm{Cont}(Y, \xi ) \quad , \quad S \mapsto \tau_S 
\end{align*}
where $\mathrm{Emb}_{\mathrm{tight}}(S^2 , (Y, \xi ) )$ stands for the space of \textit{smooth} co-oriented embeddings $S^2 \subset Y$ which admit a tight neighbourhood. In particular, if $(Y, \xi )$ is tight (globally) then $\tau_S $ only depends up to contact isotopy on the \textit{smooth} isotopy class of the co-oriented sphere $S$.

The following particular case will play an essential role in this article, so we emphasize it now. Consider a Darboux ball $B = \phi (\mathbb{B}^3 )$ in a contact manifold $(Y , \xi )$. Associated to an exterior sphere (i.e. contained in the complement $Y \setminus B$) parallel to $\partial B$ we have a well defined contact Dehn twist which fixes $B$ pointwise. By abuse in notation and for convenience we denote this contactomorphism by $\tau_{\partial B}$ even if the Dehn twist is not on the sphere $\partial B$. This defines a map of sets
\begin{align*}
    \pi_0 \mathrm{Emb}\big( (\mathbb{B}^3 , \xi_{\mathrm{st}} ), (Y, \xi ) \big) \rightarrow \pi_0 \mathrm{Cont} (Y, \xi , B ) \quad, \quad B \mapsto \tau_{\partial B}.
\end{align*}
The following convenient description of $\tau_{\partial B}$ follows from the local calculation in the proof of Lemma \ref{lem:ContSphericalAnnulus}. 

\begin{Lemma}\label{reparamdefn}
The Dehn twist $\tau_{\partial B} \in \pi_0 \mathrm{Cont} (Y, \xi , B)$ agrees with the image of $1 \in \mathbb{Z}$ under the map $$\mathbb{Z} = \pi_1 \mathrm{U}(1) \rightarrow \pi_1 \mathrm{Emb} ( ( \mathbb{B}^3 , \xi_{\mathrm{st}} )  , ( Y, \xi ) )  \rightarrow \pi_0 \mathrm{Cont} (Y , \xi , B )$$
where the first map is induced by the reparametrisation map 
\begin{align*}
\mathrm{U}(1) \rightarrow \mathrm{Emb} \big( ( \mathbb{B}^3 , \xi_{\mathrm{st}} )  , ( Y , \xi ) \big) \quad, \quad \theta \mapsto \phi \circ r_\theta 
\end{align*}
and the second map is the connecting map in the long exact sequence of the fibration (\ref{contball}).
\end{Lemma}

Let $e\in \mathrm{CEmb} \big( S^2  , ( Y, \xi ) \big)$ be an embedding of a standard convex sphere. Thus, the image $S = e (S^2 ) \subset (Y, \xi )$ is a co-oriented standard convex sphere, let $S_\pm$ be two parallel copies of $S$ given by pushing $S$ forward and backward. By the local calculation in Lemma \ref{lem:ReparametrizationLoop} we have:

\begin{Lemma}\label{reparamLemma}
The product of Dehn twists $(\tau_{S_-})^{-1} \tau_{S_+} \in \pi_0 \mathrm{Cont} (Y, \xi , S )$ agrees with the image of $1 \in \mathbb{Z}$ under the map
$$\mathbb{Z} = \pi_1 \mathrm{U}(1) \rightarrow \pi_1 \mathrm{CEmb} \big( S^2  , ( Y, \xi ) \big)  \rightarrow \pi_0 \mathrm{Cont} (Y , \xi , S )$$
where the first map is induced by the reparametrisation map 
\begin{align*}
\mathrm{U}(1) \rightarrow \mathrm{CEmb} \big( S^2  , ( Y , \xi ) \big) \quad, \quad \theta \mapsto e \circ r_\theta 
\end{align*}
and the second map is the connecting map in the long exact sequence of the fibration (\ref{contsphere}).
\end{Lemma}

\subsection{The Dehn twist and the evaluation map}

We move on to study a \textit{relative} version of the isotopy problem for the Dehn twist. Consider the Dehn twist $\tau_{\partial B}$ on (an exterior sphere parallel to) the boundary $\partial B$ of a Darboux ball, as in the previous section. We will now rephrase the problem of whether $\tau_{\partial B}^2 $ defines the trivial class in $ \pi_0 \mathrm{Cont}_0(Y, \xi , B )$ as a \textit{lifting} problem.

\subsubsection{The obstruction class}

The main player is the evaluation mapping $ev_B : \mathcal{C}(Y , \xi ) \rightarrow S^2$ defined by (\ref{ev}), which is a fibration (Lemma \ref{evfib}). If $\delta : \pi_2 S^2 \rightarrow \pi_1 \mathcal{C}(Y , \xi , B )$ is the connecting map in the homotopy long exact sequence, then we have a distinguished class 
\begin{align}
    \mathcal{O}_{\xi} := \delta ( 1)  \in \pi_1 \mathcal{C}(Y , \xi , B ) \label{obstr}
\end{align}
which, by construction, is the \textit{obstruction class} to finding a homotopy section of $ev_B$ (i.e. a map $s : S^2 \rightarrow \mathcal{C}(Y , \xi )$ such that $ev_B \circ s : S^2 \rightarrow S^2$ has degree one):
\begin{center}
$ev_B $ \textit{admits a homotopy section if and only if} $ \mathcal{O}_\xi = 0 $ .
\end{center}

Later in this section we will explicitly describe a loop of contact structures that represents the obstruction class $\mathcal{O}_\xi \in \pi_1 \mathcal{C}(Y, \xi , B )$. 

We now relate the problem of finding a section of $ev_B$ to the triviality of the Dehn twist $\tau_{\partial B}^2$ as follows. Consider the connecting map $\delta^\prime: \pi_1 \mathcal{C}(Y , \xi , B ) \rightarrow \pi_0 \mathrm{Cont}_{0} (Y , \xi , B )$ of the fibration (\ref{moserB}). The key observation is the following:

\begin{Proposition}\label{moserdehn}
The class $\delta^\prime (\mathcal{O}_\xi ) \in \pi_0 \mathrm{Cont}_{0}(Y , \xi , B) $ agrees with the \textbf{squared} contact Dehn twist $\tau_{\partial B}^2$.
\end{Proposition}

\begin{proof}
Consider first the case when $(Y , \xi )$ is the contact unit ball $( \mathbb{B}^3 , \xi_{\mathrm{st}} = \mathrm{Ker}( dz + \frac{1}{2}xdy - \frac{1}{2}ydx ) )$ and $B \subset \mathbb{B}^3 $ a subball of smaller radius with center at $0$. The fibrations from \S \ref{fibrations} fit into a commuting diagram
\[ \begin{tikzcd}
\mathcal{C}(\mathbb{B}^3 , \xi_{\mathrm{st}} , B ) \arrow{r} & \mathcal{C}(\mathbb{B}^3 , \xi_{\mathrm{st}}  ) \arrow{r}{ev_B} &  S^2 \\  
\mathrm{Diff}_0 (\mathbb{B}^3 , B ) \arrow{u} \arrow{r} & \mathrm{Diff}_{ 0} (\mathbb{B}^3 ) \arrow{r} \arrow{u}  & \mathrm{Emb} ( \mathbb{B}^3  ,  \mathbb{B}^3 ) \simeq \mathrm{SO}(3) \arrow{u} \\
\mathrm{Cont}_0(\mathbb{B}^3 , \xi_{\mathrm{st}} ,  B ) \arrow{u} \arrow{r} & \mathrm{Cont}_{ 0} (\mathbb{B}^3 , \xi_{\mathrm{st}}) \arrow{r} \arrow{u}  & \mathrm{Emb} ( ( \mathbb{B}^3 , \xi_{\mathrm{st}} )  , ( \mathbb{B}^3 , \xi_{\mathrm{st}} ) )\simeq \mathrm{U}(1) \arrow{u} 
\end{tikzcd}
\]

In the third vertical fiber sequence the map $\pi_2 S^2 = \mathbb{Z} \rightarrow \pi_1 \mathrm{U}(1) = \mathbb{Z}$ is multiplication by $2$. From the diagram we see that the image of $\mathcal{O}_{\xi_{\mathrm{st}}} \in \pi_1 \mathcal{C}(\mathbb{B}^3 , \xi_{\mathrm{st}} , B )$ in $\pi_0 \mathrm{Cont}_0 (\mathbb{B}^3 , \xi_{\mathrm{st}} , B ) $ can be alternatively calculated as the image of $2 \in \mathbb{Z} = \pi_1 \mathrm{U}(1)$ in $\pi_0 \mathrm{Cont}_0 (\mathbb{B}^3 , \xi_{\mathrm{st}} , B )$. From Lemma \ref{reparamdefn} this is the class of $\tau_{\partial B}^2$.

For an arbitrary $(Y , \xi )$ and a Darboux ball $B \subset Y$ the result then follows from the previous local calculation by extending the contact embedding $B \hookrightarrow Y$ to a contact embedding $B \subset \mathbb{B}^3 \hookrightarrow Y$, and considering the commuting diagram
\[ \begin{tikzcd}
\pi_2 S^2 \arrow{r} \arrow{d}  & \pi_1 \mathcal{C}( \mathbb{B}^3 , \xi_{\mathrm{st}} , B ) \arrow{r} \arrow{d}  &  \pi_0 \mathrm{Cont}_{0}(\mathbb{B}^3 , \xi_{\mathrm{st}} , B) \arrow{d} \\
\pi_2 S^2 \arrow{r}    & \pi_1 \mathcal{C}( Y , \xi , B) \arrow{r}  &  \pi_0 \mathrm{Cont}_{0}(Y , \xi , B) . 
\end{tikzcd}
\]
\end{proof}

\begin{Corollary}\label{criterionB}
Suppose $Y$ is aspherical (i.e. irreducible and with infinite fundamental group). Then $\tau_{\partial B}^2$ is isotopic to the identity rel. $B$ if and only if the evaluation mapping (\ref{ev}) admits a homotopy section (i.e. the obstruction class $\mathcal{O}_\xi$ vanishes).
\end{Corollary}

\begin{proof}
By the fibration (\ref{moserB}) we have the exact sequence
\begin{align*}
    \pi_1 \mathrm{Diff}_0 (Y, B ) \rightarrow \pi_1 \mathcal{C}(Y, \xi , B ) \rightarrow \pi_0 \mathrm{Cont}_0 (Y, \xi , B )
\end{align*}
so by Proposition \ref{moserdehn} the result will follow from $\pi_1 \mathrm{Diff}_0(Y, B ) = 0$. Let us now explain why the latter group vanishes. By the fibration (\ref{diffball}) we have an exact sequence
\begin{align*}
     1 \rightarrow \pi_1 \mathrm{Diff}_0(Y , B ) \rightarrow \pi_1 \mathrm{Diff}_0(Y  ) \rightarrow \pi_1 \mathrm{Fr}(Y ) \cong \pi_1 Y \times \mathbb{Z}_2 .
\end{align*}
Here, to have a $1$ on the left we use $\pi_2 Y = 0$ (which follows from $Y$ being aspherical). Since the homomorphism $\pi_1 \mathrm{Diff}_0(Y  ) \rightarrow  \pi_1 Y$, and hence $\pi_1 \mathrm{Diff}_0(Y  ) \rightarrow \mathrm{Fr}(Y)$, is injective it follows by exactness that $\pi_1 \mathrm{Diff}_0(Y , B )=0$ as required. The fact that $\pi_1 \mathrm{Diff}_0 (Y) \rightarrow \pi_1 Y$ is injective follows from the calculation of the homotopy type of the group $\mathrm{Diff}_0 (Y)$ for all aspherical\footnote{For the irreducible $3$-manifolds with finite fundamental group, the calculation of the homotopy type of $\mathrm{Diff}_0 (Y)$ has also been completed \cite{elliptic,bk1,bk2}. Thus, the homotopy type of $\mathrm{Diff}_0 (Y)$ is known for all prime $3$-manifolds.} $3$-manifolds. More precisely, the papers \cite{hatcher1, hatchersmale,S1xS2,gabai,ivanov,mcsoma} cover all aspherical $3$-manifolds with the exception of the non-Haken infranil manifolds (see \cite{mcsoma} for a nice summary). The latter consist of the non-trivial $S^1$-bundles over $T^2$, which are covered by \cite{bk3}. In all these cases $\mathrm{Diff}_0 (Y)$ has the homotopy type of $(S^1)^k$ where $k$ is the rank of the center of $\pi_1 Y$ and $\pi_1 \mathrm{Diff }_0(Y) \rightarrow \pi_1 Y$ is the inclusion of the center. The proof is now complete.
\end{proof}

In the local model $(Y, \xi ) = (\mathbb{B}^3 , \xi_{\mathrm{st}} = \mathrm{Ker}(dz + \frac{1}{2}xdy - \frac{1}{2}ydx ) )$, and letting $B \subset \mathbb{B}^3$ be any concentric sub-ball, we have the following unique characterisation of the obstruction class:

\begin{Lemma}\label{Lemma:Olocal}
The evaluation  of contact structures on $\mathbb{B}^3$ along the radial line $\{ (0 , 0 , z) \, | \, z \in [0,1] \} \subset \mathbb{B}^3$ identifies the evaluation fibration on $\mathcal{C}(\mathbb{B}^3 ,\xi_\mathrm{st} )$ with the path fibration on $S^2$:
\[
\begin{tikzcd}
\mathcal{C}(\mathbb{B}^3 , \xi_\mathrm{st} , B ) \arrow{r} \arrow{d}{\simeq} & \mathcal{C}(\mathbb{B}^3 , \xi_\mathrm{st} ) \arrow{r}{ev_B} \arrow{d}{\simeq} & S^2 \arrow{d}{=}\\
\Omega S^2 \arrow{r} & P S^2 \arrow{r} & S^2.
\end{tikzcd}
\]
Thus, the obstruction class $\mathcal{O}_{\xi_{\mathrm{st}} } \in \pi_1 \mathcal{C}(\mathbb{B}^3, \xi_{\mathrm{st}} , B )$ corresponds to the standard generator of $\pi_1 \Omega S^2 = \pi_2 S^2 $.\end{Lemma}

\begin{proof}
By the Eliashberg--Mischachev Theorem \cite{EM}, the space $\mathcal{C}(\mathbb{B}^3 , \xi_\mathrm{st} )$ is contractible. So is the path space $PS^2$, so the desired result follows.
\end{proof}

\subsubsection{Geometric description of the obstruction class}

It is instructive to describe an explicit loop $ (\xi_\varphi )_{\varphi \in S^1}$  of contact structures on $Y$ fixed over a Darboux ball $B \subset Y$ which represents the obstruction class $\mathcal{O}_\xi \in \pi_1 \mathcal{C}(Y, \xi , B )$, and we do this now. However, we won't use this construction in the remainder of the article.

By definition of the connecting map $\delta : \pi_2 S^2 \rightarrow \pi_1 \mathcal{C}(Y, \xi, B )$ associated to the fibration (\ref{ev}), a based loop $\xi_{\varphi} \in \mathcal{C}(Y, \xi , B )$ represents the obstruction class $\mathcal{O}_{\xi}$ precisely when there exists a $\mathbb{B}^2$-family of contact structures $\xi_{r, \varphi} \in \mathcal{C}(Y, \xi )$ (here, the unit $2$-ball $\mathbb{B}^2$ is parametrised using polar coordinates $(r, \varphi )$ ) with $\xi_{r, 0 } = \xi$ such that $\xi_{1, \varphi} = \xi_{\varphi}$ and the mapping $\mathbb{B}^2 \ni (r, \varphi ) \mapsto ev_B (\xi_{r,\varphi} ) = \xi_{r , \varphi} (p ) \in S^2$ induces a \textit{degree one} mapping $\mathbb{B}^2 / \partial \mathbb{B}^2 \rightarrow S^2$ (note that the map out of $\mathbb{B}^2 / \partial \mathbb{B}^2$ is well-defined because $\xi_{1 , \varphi} (p )$ is constant in $\varphi$). 

Such a family $\xi_{r, \varphi}$ can be constructed as follows. First, it suffices to consider the case $(Y, \xi ) = (\mathbb{B}^3 , \xi_{\mathrm{st}} )$, and construct a family $\xi_{r, \varphi} \in \mathcal{C}(\mathbb{B}^3 , \xi , B )$ as above (and with $\xi_{r,\varphi} = \xi_{\mathrm{st}}$ near $\partial \mathbb{B}^3$). For this, choose any smooth mapping $q : [0,1] \times S^1 \rightarrow \mathrm{SO}(3)/\mathrm{U}(1)$ with 
\[
q(0 , \varphi ) = q (1, \varphi ) = [\mathrm{id}]  \, , \quad q (r, 0 ) = q(r, 2 \pi ) = [\mathrm{id}] 
\]
such that the induced map $\Sigma S^1 \rightarrow \mathrm{SO}(3) / \mathrm{U}(1)$ has degree one. Here $\Sigma S^1 \approx S^2$ is the reduced suspension of $S^1$, and in what follows we regard $S^1$ as $\mathbb{R}/2 \pi \mathbb{Z}$. Next, regarding $q$ as a homotopy of based maps $S^1 \rightarrow \mathrm{SO}(3) / \mathrm{U}(1)$, we may lift it along the fibration $\mathrm{SO}(3) \rightarrow \mathrm{SO}(3)/\mathrm{U}(1)$ to produce a family of matrices $A_{r, \varphi} \in \mathrm{SO}(3)$ parametrised by $(r, \varphi ) \in [0,1] \times S^1$ such that 
\begin{align}
[A_{r, \varphi }] = q(r, \varphi ) \, , \quad A_{0 , \varphi } = \mathrm{id} \, \quad A_{r , 0 } = A_{r, 2 \pi } = \mathrm{id}. \label{condition}
\end{align}

Because of the second condition in (\ref{condition}), it follows that $A_{r , \varphi}$ is, in fact, a family of matrices parametrised by the unit $2$-ball $\mathbb{B}^2 \cong [0 , 1] \times S^1 / 0 \times S^1$. At this point, we would like to take the $\mathbb{B}^2$-family of contact structures on $\mathbb{B}^3$ given by $\xi_{r , \varphi} = (A_{r, \varphi})_\ast \xi_{\mathrm{st}}$. By the first condition in (\ref{condition}) we then have $A_{1, \varphi} \in \mathrm{U}(1)$, and from this it follows that $\xi_{1 , \varphi}$ is a loop of contact structures on $\mathbb{B}^3$ fixed over the ball $B \subset \mathbb{B}^3$ (in fact this loop is constant everywhere on $\mathbb{B}^3$, $\xi_{1 , \varphi} = \xi_{\mathrm{st}}$) and the induced map $\mathbb{B}^2 / \partial \mathbb{B}^2 \rightarrow S^2$ given by $(r, \varphi ) \mapsto \xi_{r, \varphi}(0 )$ has degree one, as required. However, the $\mathbb{B}^2$-family $\xi_{r , \varphi}$ is not constant near the boundary of $\mathbb{B}^3$, so we must appropriately "cut off" this family near the boundary.

We can do this as follows. Introduce a smooth cutoff function $\beta$ on $ \mathbb{B}^3$ which is identically one over $B \subset \mathbb{B}^3$ and vanishes near $\partial \mathbb{B}^3$. For each $(r , \varphi ) \in [0,1] \times S^1$ we consider the following vector field supported in the interior of $\mathbb{B}^3$ 
\[
V_{r, \varphi} (x) = \beta (x) \frac{\partial}{\partial r} A_{r, \varphi} \cdot x
\]
We regard $V_{r, \varphi}$ as a $\varphi$-family of $r$-dependent vector fields on $\mathbb{B}^3$, and consider the associated $\varphi$-family of flows $\Phi_{\varphi}^r : \mathbb{B}^3 \rightarrow \mathbb{B}^3$ starting at time $r = 0$, namely
\[
\frac{\partial}{\partial r} \Phi_{\varphi}^r (x) = V_{r, \varphi} ( \Phi_{ \varphi}^{r} (x) ) \, , \quad \Phi_{ \varphi}^{0} (x) = x \, .
\]
Over the ball $B \subset \mathbb{B}^3$ we have, by construction, that $\Phi_{\varphi}^r (x) = A_{r,\varphi} \cdot x$, and near the boundary of $\mathbb{B}^3$ we have $\Phi_{r, \varphi} = \mathrm{id}$. Hence, the $\mathbb{B}^2$-family of contact structures defined by $\xi_{r, \varphi} := (\Phi_{\varphi}^r)_\ast \xi_{\mathrm{st}}$ is now constant near the boundary of $\mathbb{B}^3$, and still has the required properties.

The loop $\xi_{1 , \varphi}$ of contact structures in $\mathcal{C}(\mathbb{B}^3 , \xi_{\mathrm{st}} , B )$ thus constructed is then an explicit representative of the obstruction class $\mathcal{O}_\xi \in \pi_1 \mathcal{C}(\mathbb{B}^3 , \xi_{\mathrm{st}} , B )$. This loop can then be implanted into an arbitrary contact $3$-manifold $(Y, \xi)$ along a Darboux chart $(\mathbb{B}^3 , \xi_\mathrm{st}) \subset (Y, \xi )$ to give a representative of the obstruction class for arbitrary $(Y , \xi )$.

\subsection{Formal triviality of $\tau_{\partial B}^2$}

We continue in the setting of the previous section, and we show

\begin{Lemma}\label{ftrivialtwist}
Suppose the Euler class of $\xi$ vanishes. Then both the loop of contact structures given by the obstruction class $\mathcal{O}_\xi \in \pi_1 \mathcal{C}(Y, \xi , B )$ and the squared Dehn twist $\tau_{\partial B}^2 \in \pi_0 \mathrm{Cont}_0 (Y, \xi , B )$ are formally trivial rel. $B$.
\end{Lemma}

\begin{proof}
On the space of co-oriented plane fields we have an analogous evaluation mapping (a fibration also, in fact)

$$\Xi (Y , \xi ) \rightarrow S^2 \quad , \quad \xi^\prime \mapsto \xi^\prime (0 ) .$$

When the Euler class of $\xi$ vanishes then we may identify $\Xi (Y, \xi )$ with the space $\mathrm{Map}_0 (Y_ , S^2 )$ of null-homotopic smooth maps $Y \rightarrow S^2$. The evaluation mapping becomes identified with the obvious evaluation mapping on this latter space. Clearly this fibration admits a section given by the constant maps $Y \rightarrow S^2$. Hence, the corresponding obstruction class vanishes, and hence $$ \mathcal{O}_\xi \in \mathrm{Ker} \big( \pi_1 \mathcal{C}(Y, \xi , B ) \rightarrow \pi_1 \Xi (Y, \xi , B ) \big)$$ 
so $\mathcal{O}_\xi$ is formally trivial. From the rel. $B$ analogue of Corollary \ref{comp} it follows that $\tau_{\partial B}^2$ is formally trivial also. \end{proof}

\subsection{Behaviour of $\mathcal{O}_\xi$ under sum}

We proceed by discussing how the obstruction class $\mathcal{O}_\xi$ from (\ref{obstr}) interacts with formation of connected sums. 

First, we briefly review a convenient model for the contact connected sum \cite{colin,geiges}. We write $(Y_0,\xi_0)=([-1,1]\times S^2,\ker(zds+\frac{1}{2}xdy-\frac{1}{2}ydx))$. Let $(Y_{\pm} , \xi_{\pm} ) $ be two contact $3$-manifolds with Darboux balls $B_\pm   \subset Y_\pm$, and coordinates $\phi_{\pm}:(Y_0,\xi_0)\hookrightarrow (Y_\pm,\xi_\pm)$ around $\partial B_{\pm}$ such that $\phi_{\pm}(\{0\}\times S^2)=\partial B_\pm$ and $\phi_\pm(Y_0)\cap Y_\pm\backslash B_\pm=\phi_\pm((0,1]\times S^2)$. Consider the smaller ball $B^0_\pm=B_\pm\backslash \phi_\pm(Y_0) \subset B_\pm$. To define the contact connected sum we use the gluing contactomorphism $G$ of $(Y_0,\xi_0)$ given by $G(s,x,y,z)=(-s,-x,-y,-z)$. 
\begin{Definition}\label{connectedsum}
The \textit{connected sum} of contact manifolds $$(Y_\# , \xi_\# ) = (Y_- , \xi_- ) \# (Y_+ , \xi_+ ) $$ is defined is defined to be $(Y_-\backslash B^0_-,\xi_-)\cup_G (Y_+\backslash B^0_+,\xi_+)$.
\end{Definition}

The connected sum of contact manifolds is well-defined and independent of choices up to contactomorphism \cite{colin}.


We will fix a Darboux ball $B_\# \subset Y_\#$ inside the neck region $[-1,1]\times  S^2=\phi_-(Y_0)=\phi_+(Y_0) \subset Y_\#$. We also have natural inclusions $\mathcal{C}(Y_{\pm} , \xi_{\pm} , B_{\pm} ) \subset \mathcal{C}(Y_\# , \xi_\# , B_\# )$. We consider their induced maps on $\pi_1$
\begin{align*}
   (-) \# \xi_+ : \pi_1 \mathcal{C}(Y_- , \xi_- , B_- ) \rightarrow \pi_1 \mathcal{C}(Y_\# , \xi_\# , B_\# )\\
    \xi_- \# (-) : \pi_1 \mathcal{C}(Y_+ , \xi_+ , B_+ ) \rightarrow \pi_1 \mathcal{C}(Y_\# , \xi_\# , B_\# )
\end{align*}

\begin{Proposition}\label{obstrsum}
The obstruction class $\mathcal{O}_{\xi_\#} \in \pi_1 \mathcal{C}(Y_\#, \xi_\# , B_\# )$ is given by $$\mathcal{O}_{\xi_\#} = ( \mathcal{O}_{\xi_-} \# \xi_+ ) \cdot ( \xi_- \# \mathcal{O}_{\xi_+} ) .$$
\end{Proposition}

\begin{proof}
It suffices to prove the corresponding statement in the local model where $(Y_\pm , \xi_\pm) = (\mathbb{B}^3 , \xi_\mathrm{st} ) $, $B_\pm \subset Y_\pm$ are smaller concentric sub-balls and $(Y_\# , \xi_\# ) = (Y_- \setminus B_-, \xi_- ) \cup_{\partial B_- = - \partial B_+} (Y_+ , \xi_+ ).$ 
We let $B_\# \subset Y_\#$ be a Darboux chart contained in the interior. 

Consider the map $j : \mathcal{C}(Y_- , \xi_- , B_- ) \times \mathcal{C}(Y_+ , \xi_+ , B_+ ) \rightarrow \mathcal{C}(Y_\#, \xi_\# , B_\# )$ given by gluing together the contact structures on the pieces $Y_\pm \setminus B_\pm$ to produce a contact structure on $Y_\#$. Choose paths $\gamma_\pm \subset Y_\#$ which connect the ball $B_\#$ to the component $\partial Y_\pm $ of the boundary of $Y_\#$. There is a commuting diagram
\[
\begin{tikzcd}
\mathcal{C}(Y_\#, \xi_\#, B_\# )  \arrow{r}{ev_{\gamma_-} \times ev_{\gamma_+}} & (\Omega S^2)^2\\
\mathcal{C}(Y_- , \xi_- , B_- ) \times \mathcal{C}(Y_+ , \xi_+ , B_+ ) \arrow{u}{j} \arrow{r}{\simeq}  & (\Omega S^2)^2 \arrow{u}{=}
\end{tikzcd}
\]
where the bottom horizontal map is given by the homotopy equivalences of Lemma \ref{Lemma:Olocal}. By a similar argument to that in the proof of Lemma \ref{Lemma:Olocal} one shows that the top horizontal map is a homotopy equivalence. Thus $j$ is also a homotopy equivalence.

Next, consider the following diagram of spaces and maps, where the bottom row is two copies of the path fibration over $S^2$, and $\Delta$ is the diagonal map.
\[
\begin{tikzcd}
\mathcal{C}(Y_\# , \xi_\# , B_\# ) \arrow{d}{ev_{\gamma_-} \times ev_{\gamma_+}} \arrow{r}  &  \mathcal{C}(Y_\# , \xi_\# ) \arrow{d}{ev_{\gamma_-} \times ev_{\gamma_+}} \arrow{r} {ev_{B_\#} } & S^2 \arrow{d}{\Delta } \\ 
(\Omega S^2 )^2 \arrow{r} & (P S^2 )^2 \arrow{r} &(S^2)^2 .
\end{tikzcd}
\]

The path fibration on $S^2$ can be identified with the evaluation fibration on $\mathcal{C}(Y_\pm, \xi_\pm )$ at the center of $B_\pm$, by Lemma \ref{Lemma:Olocal}. Under this identification, the left-most vertical map in the second diagram becomes the homotopy inverse of the map $j$ (which follows from the first diagram). Combining these observations and passing to the long exact sequence in homotopy groups in the second diagram yields a commutative square
\[
\begin{tikzcd}
\pi_2 S^2 \arrow{r}{\delta} \arrow{d} &   \pi_1 \mathcal{C}(Y_\# , \xi_\# , B_\# )  \\
\pi_2 S^2 \times \pi_2 S^2  \arrow{r}{\delta \times \delta} & \pi_1 \mathcal{C}(Y_{-}\setminus B_- , \xi_-  ) \times \pi_1 \mathcal{C}(Y_{+} \setminus B_+ , \xi_+ ) \arrow{u}{j} \end{tikzcd}
\]
from which the desired result follows.
\end{proof}

\begin{Remark} In particular, it follows from Propositions \ref{obstrsum} and \ref{moserdehn} that we have the relation $\tau_{\partial B_+}^2 \tau_{\partial B_-}^2 = \tau_{B_\#}^2$ in $\pi_0 \mathrm{Cont}_0 ( Y_\# , \xi_\# ,  \partial B_\# )$. 
\end{Remark}


\subsection{Examples: trivial Dehn twists}

For comparison with Theorem \ref{mainthm} we now exhibit examples where the squared Dehn twist on a connected sum becomes trivial as a contactomorphism.

\subsubsection{Quotients of $S^3$}
 Let $\Gamma$ be a finite subgroup of $\mathrm{U}(2)$. Then $\Gamma$ preserves the standard contact structure $\xi_{\mathrm{st}} = \mathrm{ker}(\sum_{j = 1,2} x_j dy_j - y_j dx_j )$ on the unit $3$-sphere $S^3$, so it descends onto the quotient $M_\Gamma = S^3 /\Gamma$. The $M_\Gamma$'s are the spherical $3$-manifolds and include, among others, the lens spaces $L(p,q)$ and the Poincaré sphere $\Sigma(2,3,5)$.

\begin{Lemma}\label{quotients}
The squared Dehn twist $\tau_{\partial B}^2$ on the boundary of a Darboux ball $B \subset M_\Gamma$ is contact isotopic to the identity rel. $B$. Hence the squared Dehn twist $\tau_{S_\#}^2$ on the separating sphere $S_\#$ in $(Y , \xi ) \# ( M_\Gamma , \xi_{\mathrm{st}} )$ is contact isotopic to the identity.
\end{Lemma}
\begin{proof}
The center of $\mathrm{U}(2)$ is given by the subgroup $\cong \mathrm{U}(1)$ of diagonal matrices with diagonal $(\lambda , \lambda )$ for some $\lambda \in \mathrm{U}(1)$. This subgroup acts on $M_\Gamma$ by contactomorphisms and thus also on the space of Darboux balls, which is homotopy equivalent to $M_\Gamma \times \mathrm{U}(1)$ by \ref{embfr}. This gives a map $\pi_1 \mathrm{U}(1) = \mathbb{Z} \rightarrow \pi_1 (M_\Gamma \times \mathrm{U}(1) ) = \Gamma \times \mathbb{Z}$ which we assert is given by $1 \mapsto ( e ,  2 )$ where $e \in \Gamma$ is the identity element. From Lemma \ref{reparamdefn} and this assertion, the result would follow.

That the component $\mathbb{Z} \rightarrow \Gamma$ is trivial follows from $\mathrm{U}(1)$ being the center of $\mathrm{U}(2)$. To verify that $\mathbb{Z} \rightarrow \mathbb{Z}$ is multiplication by $2$ we need to calculate the change in contact framing under the action of $\mathrm{U}(1)$. We view $S^3$ as the unit sphere in the quaternions $\mathbb{H} = \mathbb{R}\langle 1, i , j , k \rangle$, so the tangent space at $q \in S^3$ is given by $T_q S^3 = \mathbb{R}\langle i q , jq , kq \rangle$ and the standard contact structure is $\xi_{\mathrm{st}}(q) = \mathbb{R}\langle jq , kq \rangle = \mathbb{C}\langle jq \rangle$. Thus, the frame $jq$ trivializes $\xi_{\mathrm{st}} \cong \mathbb{C}$ as a complex line bundle. The center subgroup $\mathrm{U}(1) \subset \mathrm{U}(2)$ acts on $S^3$ by $( \lambda , q  ) \mapsto \lambda q$, and the action of $\mathrm{U}(1)$ on the frame $jq$ is 
$$ \lambda \cdot jq =  j \overline{\lambda} q = \lambda^2 \cdot j (\lambda q )  
$$ 
and thus the action on $\xi_{\mathrm{st}} \cong \mathbb{C}$ is by multiplication by $\lambda^2$ on the fibres. This establishes our assertion, and hence the proof is complete.
\end{proof}

\begin{Remark}
When $\Gamma \subset \mathrm{SU}(2)$, an alternative proof of Lemma \ref{quotients} can be obtained by instead exhibiting a section of $ev_B : \mathcal{C}(M_\Gamma , \xi_{\mathrm{st}} ) \rightarrow S^2$. The point is that the radial vector field $x \partial_x + y \partial_y + z \partial_z + w \partial_w $ is a Liouville vector field for each of the symplectic forms $\omega_u$, $u \in S^2$, in the flat hyperkähler structure of $\mathbb{R}^4$. The induced $S^2$-family of contact structures $\xi_u$ on $S^3$ descends to the quotients $M_\Gamma$ (with $\Gamma \subset \mathrm{SU}(2)$) and provides a section of $ev_B$. 
\end{Remark}

\subsubsection{$S^1 \times S^2$}

Consider the unique tight contact structure on $S^1 \times S^2$, given by $\xi_0 = \mathrm{Ker}(zd\theta + \frac{1}{2} xdy - \frac{1}{2}ydx )$.
\begin{Lemma}\label{S1S2}
The squared Dehn twist $\tau_{\partial B}^2$ on the boundary of a Darboux ball $B \subset S^1 \times S^2$ is contact isotopic to the identity rel. $B$. Hence the squared Dehn twist $\tau_{S_\#}^2$ on the separating sphere $S_\#$ in any contact connected sum of the form $  (Y , \xi ) \# ( S^1 \times S^2 , \xi_{0} ) $ is contact isotopic to the identity.
\end{Lemma}
\begin{proof} Let $R_\varphi$ be the counterclockwise rotation in the $xy$ plane of angle $\varphi$. By considering the subgroup $\{F_\varphi:\varphi \in S^1\}\simeq \mathrm{U}(1) \subset \mathrm{Cont}(S^1 \times S^2 , \xi_0 )$, given by $F_\varphi(\theta, x,y,z):=(\theta, R_\varphi(x,y),z)$, one easily checks that $\pi_1(\mathrm{Cont}(S^1\times S^2,\xi_0)\rightarrow \pi_1\mathrm{Emb}((\mathbb{B}^3,\xi_\mathrm{st}),(S^1\times S^2,\xi_0))\rightarrow \pi_1\mathrm{U}(1)$ is surjective, so the result follows. 
\end{proof}

\begin{Remark}
In turn, the contact Dehn twist on the non-trivial sphere in $( S^1 \times S^2, \xi_0 )$ is non-trivial (and with infinite order). However, it is formally non-trivial already and therefore not exotic, see \S \ref{Gompfsection}.
\end{Remark}

\subsubsection{Sum with an overtwisted contact $3$-manifold}

Let $(r,\theta,z)\in\mathbb{R}^3$ be cylindrical coordinates. Consider the contact structure $\xi_{\mathrm{ot}}$ in $\mathbb{R}^3$ defined by the kernel of 
$$ \alpha_{\mathrm{ot}}= \cos r dz+r\sin r d\theta. $$

The disk $\Delta_{\mathrm{ot}}=\{(r,\theta,z)\in\mathbb{R}^3:z=0, r\leq \pi\}$ is an \em  overtwisted disk.\em 

\begin{Definition}[Eliashberg \cite{EliashbergOT}]
An overtwisted contact $3$-manifold is a contact $3$-manifold that contains an embedded overtwisted disk. 
\end{Definition}

Let $\mathcal{C}(Y, \Delta_{\mathrm{ot}})$ be the space of contact structures in $Y$ with a fixed overtwisted disk $\Delta_{\mathrm{ot}}\subset Y$. Let $\Xi (Y , \Delta_{\mathrm{ot}} )$ be the space of co-oriented plane fields in $Y$ tangent to $\Delta_{\mathrm{ot}}$ at the point $0 \in \Delta_{\mathrm{ot}}$. A foundational result of Eliashberg, generalised in higher dimensions by Borman, Eliashberg and Murphy, is 

\begin{Theorem}[Eliashberg \cite{EliashbergOT,BEM} ]\label{thm:HPrincipleOT}
The inclusion $$\mathcal{C}(Y , \Delta_{\mathrm{ot}})\rightarrow \Xi (Y , \Delta_{\mathrm{ot}} )$$ is a homotopy equivalence. 
\end{Theorem}
\begin{Remark}\label{rmk:RelativeH-Principle}
A relative version Eliashberg's h-principle is available. Suppose $A\subseteq Y\setminus \Delta_{\mathrm{ot}}$ is compact and $Y\backslash A$ is connected. Given a family of co-oriented plane fields $\xi^k\in\Xi(Y , \Delta_{\mathrm{ot}})$ that is contact over an open neighbourhood of $A$ there exists a homotopy rel. $A$ from $\xi^k$ to a family of contact structures.
\end{Remark}

Using Eliashberg's $h$-principle we obtain

\begin{Lemma}\label{lem:OT}
Let $(Y,\xi)$ be a contact $3$-manifold with vanishing Euler class. Then, for every overtwisted contact $3$-manifold $(M,\xi_{\mathrm{ot}})$ the squared contact Dehn twist $\tau_{S_\#}^2$ in $(Y,\xi)\# (M,\xi_{\mathrm{ot}})$ is contact isotopic to the identity. 
\end{Lemma}
\begin{proof}
Let $B\subset (Y,\xi)$ be a Darboux ball that we remove when performing the connected sum. By Lemma \ref{ftrivialtwist} we have that $\tau_{\partial B}^{2}$ is formally contact isotopic to the identity rel. $B$. It follows that $\tau_{S_\#}^{2}$ is formally contact isotopic to the identity on $Y \# M$, in fact relative to a small ball $B_{\mathrm{ot}}$ containing an overtwisted disk $\Delta_{\mathrm{ot}}\subset M$. At this point, by Eliashberg's Theorem \ref{thm:HPrincipleOT} and Lemma \ref{moserflemma} applied to the contact $3$-manifold with convex boundary $(Y\# (M\setminus B_{\mathrm{ot}}) ,\xi\#\xi_{\mathrm{ot}})$ we see that the group of contactomorphisms fixing $\Delta_{\mathrm{ot}}$ is homotopy equivalent to the corresponding space of formal contactomorphisms. The result now follows.
\end{proof}

In \S\ref{OTsection} we will see that Lemma \ref{lem:OT} implies exotic $1$-parametric phenomena in overtwisted contact $3$-manifolds.

\subsection{The Reidemeister I Move and Gompf's Contactomorphism}\label{Gompfsection}

We now describe the contact Dehn twist diagrammatically by means of front projections of Legendrian arcs. This approach is in the spirit of Gompf's description \cite{gompf} of the contact Dehn twist. For convenience we consider the unit ball $(\mathbb{B}^3,\xi=\ker (dz-ydx))$. Let $Y_0 = [-1,1]\times S^2$ be the complement in $\mathbb{B}^3$ of a small open ball $B_{\varepsilon }$ around the origin.
Consider the standard Legendrian arc $l:[-1,1]\rightarrow \mathbb{B}^3 \, , \, t\mapsto(t,0,0).$ Perform two Reidemeister I moves to the Legendrian $l$ to obtain a second Legendrian arc $\hat{l}$. We may assume that $\hat{l}$ coincides with $l$ over the $B_\varepsilon$. The front of these arcs are depicted in Figure \ref{fig:ReidemeisterI}. 

\begin{figure}[h!]
\centering
 \includegraphics[scale=0.8]{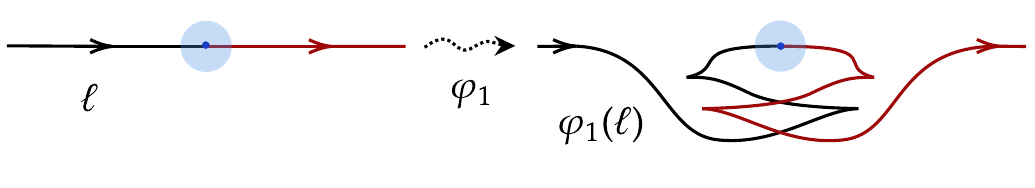}
 \caption{Front projection of $l$ and $\hat{l}$. The blue ball represents the small ball $B_\varepsilon \subset \mathbb{B}^3$.}
 \label{fig:ReidemeisterI}
\end{figure}

These arcs are Legendrian isotopic, so there exists a contact isotopy $\varphi_t\in\mathrm{Cont}(\mathbb{B}^3,\xi)$ with $\varphi_0 = \mathrm{id}$ and $\varphi_1 \circ l=\hat{l}$. Moreover, $\varphi_1$ can be taken to be the identity over $B_\varepsilon$. Therefore, $\varphi_1$ gives a contactomorphism $\tau$ of the contact manifold with convex boundary $(Y_0,\xi)$. From now on, we will denote the restrictions of $l$ and $\hat{l}$ to the red segments in Figure \ref{fig:ReidemeisterI} by the same letters for convenience. We have $\tau (l)= \hat{l}$ and the arc $\hat{l}$ is obtained in $(Y_0,\xi)$ from $l$ by a positive stabilization, see Figure \ref{fig:ContactDehnTwist}. In particular, $$\mathrm{rot}(\tau(l))=\mathrm{rot}(l)+1.$$

 \begin{figure}[h!]
 \centering
 \includegraphics[scale=0.8]{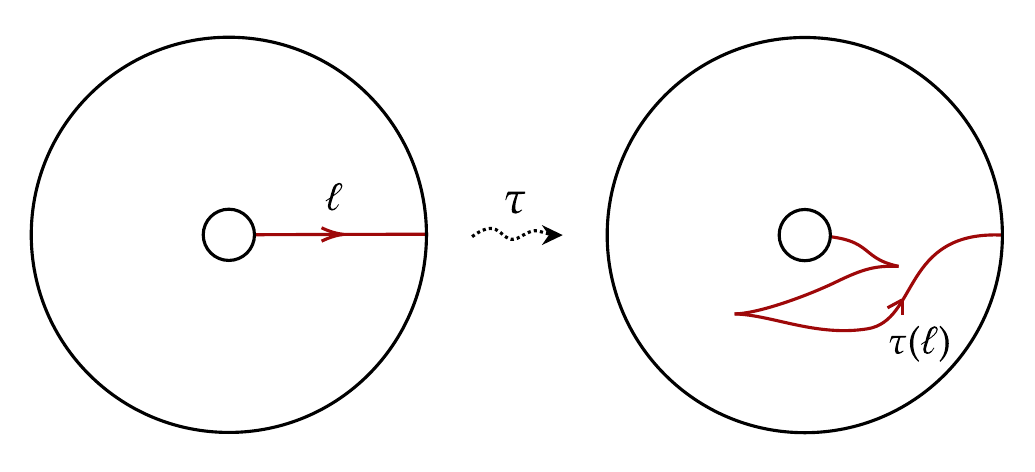}
 \caption{The image of $l$ under $\tau$.}
 \label{fig:ContactDehnTwist}
 \end{figure}
 
It follows that $\tau$ is not (formally) contact isotopic to the identity as a contactomorphism of $(Y_0,\xi)$ rel. $\partial Y_0$. This contactomorphism is contact isotopic to the contact Dehn twist as we have defined it in this section. In fact, as we will see in Lemma \ref{lem:DehnR1Move}, since the complement of $l$ is a tight $3$-ball any contactomorphism of $(Y_0,\xi)$ can be described, up to contact isotopy, just in terms of its effect on $l$ and, therefore, just by means of front projections of Legendrian arcs. First, we observe that the path-connected components of the space $\mathrm{Leg}(Y_0,\xi)$ of Legendrian embeddings of arcs that coincide with $l$ at the end points can be easily understood:
 
\begin{Lemma}
The map $\mathrm{rot}:\pi_0\mathrm{Leg}(Y_0,\xi)\rightarrow \mathbb{Z}\,, \, L\mapsto \mathrm{rot}(L),$ is an isomorphism. 
\end{Lemma}
\begin{proof} Two smoothly isotopic Legendrian arcs with the same rotation number are Legendrian isotopic after adding a finite number of double stabilizations (pairs of positive and negative stabilizations) because of the Fuchs-Tabachnikov Theorem \cite{FuchsTabachnikov}. As depicted in Figure \ref{fig:KillingDoubleStabilization}, this can done by a Legendrian isotopy in $(Y_0,\xi)$.  Therefore, the proof follows from the $3$-dimensional lightbulb Theorem.
\end{proof}

 \begin{figure}[h!]
 \centering
 \includegraphics[scale=0.8]{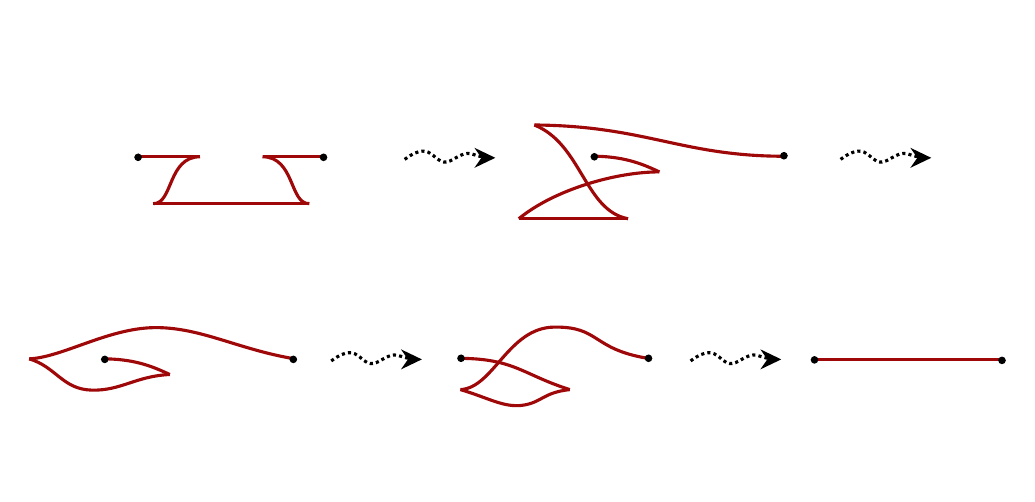} 
 \caption{Legendrian isotopy from a double stabilization of $l$ to $l$ in $(Y_0,\xi)$.}
 \label{fig:KillingDoubleStabilization}
 \end{figure}
 
We conclude the following
 
\begin{Lemma}\label{lem:DehnR1Move}
The map $\mathrm{Cont}(Y_0,\xi)\rightarrow \mathrm{Leg}(Y_0,\xi),f \mapsto f\circ l$ is a homotopy equivalence. In particular, 
$$ \pi_0 \mathrm{Cont}(Y_0,\xi)\rightarrow \mathbb{Z} \quad, \quad  f\mapsto \mathrm{rot}(f\circ l)$$ 
is an isomorphism. Moreover, the contact Dehn twist is characterized, up to contact isotopy, by the relation $$\mathrm{rot}(f(l))=\mathrm{rot}(l)+1.$$
\end{Lemma}
\begin{proof}
This follows by the previous Lemma, the Eliashberg--Mishachev Theorem \ref{thm:EM} and Hatcher's Theorem \cite{hatchersmale}, since the fiber of $\mathrm{Cont}(Y_0,\xi)\rightarrow \mathrm{Leg}(Y_0,\xi)$ can be identified with the contactomorphism group of the complement of a neighbourhood of $l$, and the latter is a tight $3$-ball.
\end{proof}

\section{Monopole Floer homology and families of contact structures} \label{monopolesection}

In this section we provide the necessary background on the Floer theoretic ingredients that come into the proof of Theorem \ref{mainthm}. For the remainder of this article, all homology groups are taken with $\mathbb{Q}$ coefficients for simplicity.

\subsection{Monopole Floer homology and the contact invariant}

For a quick introduction to Kronheimer and Mrowka's monopole Floer homology groups we recommend \cite{lin,monolens} and for a detailed treatment the monograph \cite{KM}. Here we just comment briefly on a few formal aspects.

Consider a $3$-manifold $Y$ together with spin-c structure $\mathfrak{s}$ (in this article the only spin-c structure that will be relevant is that induced by a contact structure $\xi$, denoted $\mathfrak{s}_\xi$). Associated to it there are various monopole Floer homology groups ($\mathbb{Q}$-vector spaces in this article). The ones relevant to us are the "to" and "tilde" flavors: $\HMto (Y, \mathfrak{s} )$ and $\HMtilde (Y, \mathfrak{s} )$. The former arises "formally" as the $S^1$-equivariant Morse homology of the Chern-Simons-Dirac functional. An algebraic manifestation of this equivariant nature is that $\HMto (Y, \mathfrak{s} )$ carries a module structure over the polynomial algebra $\mathbb{Q}[U] $ (i.e. the $S^1$-equivariant cohomology of a point, $\mathrm{H}_{S^1}^\bullet (\mathrm{point} ) = \mathbb{Q}[U]$) and $U$ decreases grading by two. In turn, the "tilde" flavor should be regarded as the (non-equivariant) Morse homology, and thus is an $\mathrm{H}_\bullet (S^1 ) = \mathbb{Q}[\chi]/(\chi^2 ) $-module, with $\chi$ raising degree by one. A standard \textit{Gysin sequence} relates the two groups: 
\[
\begin{tikzcd}
\cdots \arrow{r}{p} & \HMto_{\bullet}(Y , \mathfrak{s}) \arrow{r}{U} & \HMto_{\bullet-2} (-Y, \mathfrak{s}) \arrow{r}{j} & \HMtilde_{\bullet-1}(-Y , \mathfrak{s} ) \arrow{r}{p} & \cdots 
\end{tikzcd}
\]
and the map $\chi$ is recovered from this by $\chi = j p$. A common feature of all flavors of the monopole groups is a canonical grading by the set of homotopy classes of plane fields $\pi_0 \Xi (Y )$, which carries a natural $\mathbb{Z}$-action.
 
The \textit{contact invariant} $\mathbf{c}(\xi)$ is an element of $\HMto_{[\xi]} (-Y , \mathfrak{s}_\xi )$ which is well-defined up to a sign, and is canonically attached to a contact structure $\xi$ on $Y$. It was defined by Kronheimer, Mrowka, Ozsváth and Szabó in \cite{monolens}, but its definition goes back essentially to the earlier paper \cite{monocont}. Ozsváth and Szabó gave a definition of $\mathbf{c}(\xi )$ in Heegaard-Floer homology \cite{OScontact}. Under the isomorphism between the monopole and Heegaard-Floer groups \cite{KLT,CGH} the contact invariants are shown to agree. Some of the basic properties of $\mathbf{c}(\xi )$ are:
\begin{itemize}
    \item $\mathbf{c}(\xi ) = 0$ if $(Y , \xi )$ is overtwisted \cite{MR}
    \item $\mathbf{c}(\xi ) \neq 0$ if $(Y ,\xi )$ admits a strong symplectic filling \cite{mariano} 
    \item $\mathbf{c} (\xi )$ is natural under symplectic cobordisms \cite{mariano}: if $(W, \omega )$ is a symplectic cobordism $(Y_1 , \xi_1 ) \leadsto (Y_2 , \xi_2 )$ (here the convex end is $(Y_2, \xi_2)$) then 
    $$\HMto (-W , \mathfrak{s}_\omega ) \cto ( \xi_2 ) = \cto (\xi_1 )$$
    \item $U \cdot \mathbf{c}(\xi ) = 0$ (this is clear from the Heegaard-Floer point of view; in the monopole case this follows from Theorem \ref{commdiag} below).
\end{itemize}

\subsection{Families contact invariant}\label{familiessubsec}

\begin{Remark} Throughout this section we assume that $\mathbf{c}(\xi )\neq 0$ because it simplifies a little the exposition that follows (otherwise one should consider homologies with twisted coefficients, see \cite{yo}). We also resolve the sign ambiguity of $\mathbf{c}(\xi )$ by fixing one of the two. All homologies are taken with $\mathbb{Q}$ coefficients.
\end{Remark}

A version of the contact invariant for a family of contact structures was introduced by the second author in \cite{yo}. We summarize now some of those results. We have homomorphisms 
\begin{align}
    & \fc_\bullet : \mathrm{H}_\bullet ( \mathcal{C}(Y, \xi ) ) \longrightarrow \HMto_{[\xi]+\bullet}(-Y , \mathfrak{s}_{\xi}) \label{fc} \\
    & \widetilde{\fc}_\bullet : \mathrm{H}_\bullet ( \mathcal{C}(Y, \xi , B ) ) \longrightarrow \HMtilde_{[\xi]+\bullet}(-Y , \mathfrak{s}_{\xi} ). \label{fctilde}
\end{align}

The invariant $\fc_\bullet$ recovers the usual contact invariant: we have $\mathrm{H}_0 ( \mathcal{C}(Y , \xi ) ) = \mathbb{Q}$ and then $\fc_0 (1) = \mathbf{c}(\xi )$. 
The main property about them that we will exploit is the following. Associated to the fibration $ev_B : \mathcal{C}(Y , \xi ) \rightarrow S^2$ there is the Serre spectral sequence in homology. The latter collapses on the $E^3$ page and assembles into the \textit{Wang} long exact sequence:
\[
\begin{tikzcd}
\cdots \arrow{r} & \mathrm{H}_\bullet (\mathcal{C}(Y , \xi ) ) \arrow{r}{U_B}  & \mathrm{H}_{ \bullet-2} ( \mathcal{C}(Y , \xi , B ) ) \arrow{r} & \mathrm{H}_{\bullet-1} ( \mathcal{C}(Y,\xi , B ) ) \arrow{r}{\iota_\ast} & \cdots
\end{tikzcd}
\]
where $U_B$ takes the intersection of cycles in the total space of the fibration with the fiber, and $\iota_\ast$ is the map induced by inclusion $\iota : \mathcal{C}(Y, \xi, B )\rightarrow \mathcal{C}(Y , \xi )$. We note that the obstruction class $\mathcal{O}_\xi$ for $ev_B$ to admit a homotopy section arises here homologically as the image of $1$ under $\mathbb{Q} = H_0 ( \mathcal{C}(Y, \xi, B ) ) \rightarrow H_1 ( \mathcal{C}(Y, \xi , B ) )$. 

\begin{Theorem}[\cite{yo}]\label{commdiag} There is a commutative diagram (up to signs)

\begin{tikzpicture}[baseline= (a).base]
\node[scale=1, trim left=-6cm] (a) at (0,0){
\begin{tikzcd}
\cdots \arrow{r}{p} & \HMto_{[\xi] + \bullet }(-Y , \mathfrak{s}_\xi) \arrow{r}{U} & \HMto_{[\xi] + \bullet-2} (-Y ,\mathfrak{s}_\xi ) \arrow{r}{j} & \HMtilde_{[\xi]+ \bullet-1}(-Y, \mathfrak{s}_\xi ) \arrow{r}{p} & \cdots \\
\cdots \arrow{r} & \mathrm{H}_\bullet (\mathcal{C}(Y , \xi ) ) \arrow{r}{U_B} \arrow{u}{\fc_\bullet } & \mathrm{H}_{ \bullet-2} ( \mathcal{C}(Y , \xi , B ) ) \arrow{r} \arrow{u}{(\fc_{\bullet -2}) \circ \iota_\ast }& \mathrm{H}_{\bullet-1} ( \mathcal{C}(Y,\xi , B ) ) \arrow{u}{\fctilde_{\bullet-1}}\arrow{r}{\iota_\ast} & \cdots
\end{tikzcd}
};
\end{tikzpicture}
\end{Theorem}

Some observations are in order:
\begin{itemize}
\item As a particular case, Theorem \ref{commdiag} recovers a property about the contact invariant $\mathbf{c}(\xi )$ which is well-known from the Heegaard--Floer point of view: that $U \cdot \mathbf{c}(\xi ) = 0 $ and we have a canonical element $\widetilde{\mathbf{c}}(\xi ):= \fctilde_0 (1) \in \HMtilde_{[\xi]} (-Y , \mathfrak{s}_\xi )$ such that $p \widetilde{\mathbf{c}}(\xi ) = \mathbf{c}(\xi )$. Conjecturally, the invariant $\mathbf{c}(\xi )$ corresponds to the Heegaard--Floer contact invariant that takes values in $\widehat{\mathrm{HF}}(-Y, \mathfrak{s}_\xi )$, which is defined in \cite{OScontact}.

\item For two-dimensional families, Theorem \ref{commdiag} gives us the simple formula $$U \cdot \fc_2 (\beta ) = \mathrm{deg}(\beta ) \mathbf{c}(\xi )$$ where $\mathrm{deg}(\beta ) = (ev_B )_\ast \beta \in H_2 ( S^2 ) = \mathbb{Q}$ is the \textit{degree} of the family $\beta \in \mathrm{H}_2 (\mathcal{C}(Y ,\xi , B )  )$. In particular, by Theorem \ref{commdiag} we have the following

\begin{Corollary}[\cite{yo}]\label{cor:NoLagrangianRot} If $\mathbf{c}(\xi ) \notin \mathrm{Im}U$ then the fibration $ev_B$ does not admit a homotopy section and thus the obstruction class $\mathcal{O}_\xi$ is non-vanishing homologically.
\end{Corollary}

\item Other statements that are easily derived from Theorem \ref{mainthm} are:
\begin{center}
    $\mathbf{c}(\xi) \notin \mathrm{Im}U$ \textit{  if and only if  } $\fctilde_1 (\mathcal{O}_\xi ) \neq 0$
\end{center}
\begin{center}
    $\fctilde_1 (\mathcal{O}_\xi ) = \chi \widetilde{\mathbf{c}}(\xi ).$
\end{center}

\item If we define a $\mathbb{Q}[U]$-module structure on $\mathrm{H}_\bullet (\mathcal{C}(Y , \xi ) )$ by setting $U:= \iota_\ast \circ U_B$ then Theorem \ref{commdiag} asserts, in particular, that the homomorphism $\fc_\bullet : H_\bullet (\mathcal{C}(Y , \xi , B ) ) \rightarrow \HMto_{[\xi]+ \bullet} (- Y , \mathfrak{s}_\xi )$ is a map of $\mathbb{Q}[U]$-modules. Notice that we have, in fact, a $\mathbb{Q}[U]/(U^2)$-module structure on $\mathrm{H}_{2} (\mathcal{C}(Y, \xi) ) $, i.e. the action of $U^2$ on  $\mathrm{H}_\bullet ( \mathcal{C}(Y, \xi ) )$ vanishes. This can be regarded as a manifestation of the following geometric fact, that we have already encountered in \S \ref{dehnsection}. Consider two disjoint Darboux balls $B , B^\prime \subset Y$. Whereas the spaces $\mathcal{C}(Y , \xi )$ and $\mathcal{C}(Y , \xi , B  )$ are related in a possibly non-trivial way by the fibration $ev_B$, the spaces $\mathcal{C}(Y, \xi , B )$ and $\mathcal{C}(Y , \xi , B \cup B^\prime )$ are related in a straightforward way:
$$ \mathcal{C}(Y, \xi , B \cup B^\prime ) \simeq \Omega S^2 \times \mathcal{C}(Y , \xi , B )  .$$
Indeed, the evaluation map corresponding to the ball $B^\prime$ gives a fibration
$$\mathcal{C}(Y, \xi , B \cup B^\prime ) \rightarrow \mathcal{C}(Y , \xi , B ) \xrightarrow{ev_{B^\prime }} S^2 $$
but now the map $ev_{B^\prime }$ is null-homotopic, as can be seen by dragging the evaluation point (the center of $B^\prime$) into the first ball $B$. 
\end{itemize}

\subsection{Summary of the construction of the families invariants}

We summarise in this section the construction of the invariants $\fc$ and $\fctilde$, carried out in detail by the second author in \cite{yo}. This is included here for background purposes. However, the contents of this subsection will not be used later in this article.

\subsubsection{The invariant $\fc$}
We begin with some general observations. Let $X$ be a $4$-manifold together with a non-degenerate $2$-form $\omega$ i.e. $\omega^2 $ is a volume form. We use $\omega^2$ to orient $X$. Choose an almost complex structure $J$ compatible with $\omega$, which by definition gives a metric $g = \omega ( . , J . ) $. The space of choices of $J$ is contractible. The structure $J$ equips $X$ with a spin-c structure, i.e. a lift of the $\mathrm{SO}(4)$-frame bundle of $X$ along the map $\mathrm{Spin}^{c}(4 ) \rightarrow \mathrm{SO}(4)$. In differential-geometric terms this yields rank-two complex hermitian bundles $S^\pm \rightarrow X$ and Clifford multiplication $\rho : TX \rightarrow \mathrm{Hom}(S^+ , S^- )$ satisfying the "Clifford identity" $\rho(v)^\ast\rho ( v) = g(v,v) \mathrm{Id}$. We follow the notation and conventions from \S 1 in \cite{KM} and we assume the reader is familiar with these.


The Clifford action of the $2$-form $\omega$ on $S^+$ splits the bundle $S^+$ into $\mp 2i$ eigen-subbundles of rank $1$. These are given by $S^+ = E \oplus E K_{J}^{-1}$, where $K_J$ is the canonical bundle of $(X, J )$ and $E$ is a complex line bundle which is easily verified to be trivial. Choose a unit length section $\Phi_0$ of $E$. A simple calculation shows that there is a unique spin-c connection $A_0$ on $S^+$ such that $\nabla_{A_0} \Phi_0$ is a $1$-form with values in the $+2i$ eigenspace $E K_{J}^{-1}$. At this point, the symplectic condition comes in through the following calculation involving the coupled Dirac operator $D_{A_0} : \Gamma (S^+) \rightarrow \Gamma (S^- )$

\begin{Lemma}[Taubes \cite{taubessymp}]
The non-degenerate $2$-form $\omega$ is symplectic (i.e. $d \omega = 0$) if and only if $D_{A_0} \Phi_0 =0$. 
\end{Lemma}

We now bring in a smoothly varying family of symplectic structures $\omega_u $ parametrised by a smooth manifold $U \ni u$, with each $\omega_u$ in the same deformation class as $\omega$. Again, we equip the $\omega_u$'s with compatible almost complex structures $J_u$ varying smoothly, which provide us with a family of metrics $g_u$. From our original Clifford bundle $(S^\pm , \rho )$ we canonically obtain new ones as follows. The bundles $S^\pm$ remain the same but new Clifford structures $\rho_u$ are obtained by setting $\rho_u = \rho \circ b_u$ where $b_u$ is the canonical isometry $(TX , g_u ) \xrightarrow{\cong} (TX , g )$ (the unique isometry which is positive and symmetric with respect to $g_u$). The Clifford action of $\omega_u $ again decomposes $S^+$ into eigenspaces $S^+ = E_u \oplus E_u  K_{J_u}^{-1}$. Each $E_u$ is trivializable individually but the family $( E_u )_{u \in U}$ might give a non-trivial line bundle over $U \times X$. When $U$ is contractible then we may choose a family of trivialising sections $\Phi_u$ of $E_u$ with unit length, and as before these determine unique spin-c connections $A_u$ with $D_{A_u} \Phi_u = 0 $. Then, associated to our family $(\omega_u , J_u )$ and the choices of $\Phi_u$ we have a family of "deformed" Seiberg-Witten equations on $X$ given by 
\begin{align*}
    & \frac{1}{2} \rho_u (F^{+}_A )- (\Phi \Phi^\ast )_0  = \frac{1}{2} \rho_u (F^{+}_{A_u} ) - ( \Phi_u \Phi_{u}^\ast )_0  \\
    & D_{A} \Phi = D_{A_u } \Phi_u .
\end{align*}
For each $u \in U$ this is an equation on the pair $(A, \Phi )$, where $A$ is a connection on $\Lambda^2 S^+$ and $\Phi $ is a section of $S^+$. In this "deformed" version of the equations the configurations $(A_u , \Phi_u )$ solve the equation for $u$.

We apply now the above considerations to a special case. Let $(Y, \xi )$ be a closed contact $3$-manifold with a contact form $\alpha$, and let $(X, \omega ) $ be the \textit{symplectisation} $X = [1, + \infty ) \times Y$, with the exact symplectic form $\omega = d (\frac{t^2}{2} \alpha )$. The structure $J$ is chosen to be invariant under the Liouville flow, and the associated Riemannian metric on $X$ is conical. We now bring into the picture a family of contact structures $\xi_u$ parametrised by $U = \Delta^n$, to which we would like to associate an element in the Floer chain complex of $-Y = \partial X$. Here $\Delta^n$ is the standard $n$-simplex. We equip our family $\xi_u$ with corresponding contact forms $\alpha_u$. This gives a family $\omega_u$ of symplectic structures on $X$. 

The construction now proceeds by forming a manifold $Z^+$ by gluing the cylinder $Z = (- \infty , 0 ] \times Y$ with the symplectic manifold $X$. We extend all metrics $g_u$ over to $Z^+$ in such a way that they all agree with a fixed translation-invariant metric on the cylinder $Z$. Then the bundle $S^+$, together with its splitting $S^+ = E \oplus E K_{J}^{-1}$, extends over $Z^+$ naturally in a translation-invariant manner. The $U$-family of metrics and spin-c structures thus constructed on $Z^+$ are independent of $u$ over $Z$, so we have effectively trivialised our data over the cylinder end $Z \subset Z^+$. In order to extend the Seiberg-Witten equations over $Z^+$ we cut off the perturbation term on the right-hand side of the equations so that it vanishes on the cylinder end $Z$. This way, we have a $U$-parametric family of Seiberg-Witten equations over $Z^+$, and natural boundary conditions for these equations (modulo gauge) are
\begin{itemize}
    \item on the cylinder $Z$ solutions should approach a translation-invariant solution $\mathfrak{a}$ (a generator of the "to" Floer complex $\widecheck{\mathrm{C}} (- Y , \mathfrak{s}_\xi  ) $, i.e. $\mathfrak{a}$ is an irreducible or boundary stable monopole on $-Y$)
    \item on the symplectic end $X$ solutions should approach the configuration $(A_u , \Phi_u )$.
\end{itemize}
This way we obtain parametrised moduli spaces of solutions $$\pi : M ([\mathfrak{a}] , \Delta^n ) \rightarrow \Delta^n . $$ By introducing suitable perturbations we may achieve the necessary transversality \cite{yo} and $M([\mathfrak{a}] , \Delta^n ) $ will be $C^1$-manifolds of finite dimension. At this point we note that, because of the gauge-invariance of the equations, a different choice of trivialisations $\Phi_u$ would yield diffeomorphic moduli spaces. The connected components of $M ( [\mathfrak{a}] , \Delta^n )$ where the index of $\pi$ is $-n$ consist of a finite number of isolated points lying over values in the interior of $\Delta^n$, and a signed count of these points gives an integer $\# M([\mathfrak{a}] , \Delta^n ) \in \mathbb{Z}$. We organise these counts into a Floer chain $\psi (\Delta^n )$ $$\psi (\Delta^n ) = \sum_{[\mathfrak{a}]} \# M ([\mathfrak{a}] , \Delta^n ) \cdot [\mathfrak{a}]  \in \widecheck{\mathrm{C}}(- Y , \mathfrak{s}_\xi ) .$$ 
The assignment $\Delta^n \mapsto \psi (\Delta^n )$ can be made into a chain map $$\psi : \mathrm{C}_\bullet (\mathcal{C}(Y , \xi ) )  \rightarrow \widecheck{\mathrm{C}}_{\bullet}(- Y , \mathfrak{s}_\xi ) $$
from the complex of singular chains on $\mathcal{C}(Y , \xi )$. Taking homology yields the families invariant (\ref{fc}). The analytic underpinning that make all the above rigorous are discussed in \cite{yo}, and are essentially no different than those of \cite{monocont,taubesestimates}.

\subsubsection{The invariant $\fctilde$}

In terms of the "to" Floer complex $\widetilde{C}_\bullet$, the "tilde" Floer complex can be defined by taking the mapping cone of (a suitable chain level version of) the $U$ map. We have $\widetilde{C}_\bullet (Y , \mathfrak{s} ) = \widecheck{C}_{\bullet}( Y , \mathfrak{s} ) \oplus \widecheck{C}_{\bullet-1}( Y , \mathfrak{s} )$ with differential given by the matrix (ignoring signs) $$\widetilde{\partial}= \begin{pmatrix} \widecheck{\partial} & 0 \\ U & \widecheck{\partial}  \end{pmatrix}. $$

If a family $\beta \in \mathrm{H}_n ( \mathcal{C}(Y , \xi ) )$ is in the image of $\iota_\ast : \mathrm{H}_n ( \mathcal{C}(Y , \xi , B ) ) \rightarrow \mathrm{H}_n ( \mathcal{C}(Y , \xi  ) )$ then it is proved in \cite{yo} that $U \cdot \fc(\beta ) = 0$. At the chain level this is witnessed by a canonical chain homotopy $\theta$:
\begin{align}
U \cdot \psi \circ \iota_\ast  = \widecheck{\partial} \theta + \theta \partial . \label{chainid}
\end{align}
From this we build the chain map $$\widetilde{\psi} = (\psi \circ \iota_\ast , \theta ) : \mathrm{C}_\bullet ( \mathrm{C}(Y , \xi, B ) ) \rightarrow \widetilde{C}_\bullet (-Y , \mathfrak{s}_\xi )$$
which, upon taking homology gives the definition of (\ref{fctilde}). The chain homotopy $\theta$ is roughly constructed as follows. We introduce a new parameter $t \in \mathbb{R}$ and let $0 \in Y$ be the center of the ball $B$. Consider the moduli space $$\mathcal{M}([\mathfrak{a}] , \Delta^n ) \rightarrow \mathbb{R} \times \Delta^n$$ consisting of quadruples $(A, \Phi , u , t )$ such that $(A, \Phi , u )$ solve the previous set of equations and boundary conditions subject to the further constraint that at the point $(t,0 ) \in \mathbb{R} \times Y \cong Z^+$ the spinor $\Phi$ lies in the second component of the splitting $S^+ = E \oplus E K_{J}^{-1}$. By a simple modification of this construction one can again achieve transversality and ensure that the $\mathcal{M}([\mathfrak{a}], \Delta^n )$ are $C^1$-manifolds of finite dimension. Then we set
\begin{align*}
    \theta ( \Delta^n ) = \sum_{[\mathfrak{a}]} \# \mathcal{M}([\mathfrak{a}] , \Delta^n )\cdot [\mathfrak{a}].
\end{align*}

Theorem \ref{commdiag} is established by carefully analysing the "boundary at infinity" of the $1$-dimensional components of the moduli $\mathcal{M}([\mathfrak{a}], \Delta^n )$, see \cite{yo}.



\section{The space of standard convex spheres in a tight contact $3$-manifold}\label{spheressection}

In this section we provide background on an $h$-principle for standard convex embeddings in tight contact $3$-manifolds which was established in work of the first author with J. Mart\'inez-Aguinaga and F.Presas \cite{fmp}. For the sake of completeness, we will provide here a detailed account which isn't quite the same as in \cite{fmp}. 

Throughout this section $(Y, \xi )$ will be a tight contact $3$-manifold. Recall that given a contact $3$-manifold a $(Y,\xi)$, by a standard convex embedding of $S^2$ we mean a convex embedding $e:S^2 \hookrightarrow (Y,\xi)$ such that its oriented characteristic foliation $( e^* \xi )\cap TS^2$ coincides with the characteristic foliation of the sphere $$e_0:S^2\hookrightarrow \{0\}\times S^2\subset (Y_0,\xi_0)=([-1,1]\times S^2,\ker(zds+\frac 12 xdy-\frac 12 ydx)).$$

In fact, by this property we obtain a (homotopically) unique contact embedding of a neighbourhood of $e_0 (S^2 ) \subset Y_0$ inside $Y$ such that $e_0$ identifies with $e$. We recall that the north pole of $e$ is then a positive elliptic point and the south pole a negative elliptic point. See Figure \ref{fig:sphere}.

All of our arguments below work well for any other foliation of a convex sphere. The key fact is that the space of tight convex spheres with fixed characteristic foliation is $C^0$-dense inside the space of smooth spheres when the contact $3$-manifold is tight, because of Giroux's Genericity and Realisation Theorems and Giroux's Tightness Criterion \cite{convexite,GirouxCriterion}. 

\begin{Remark} The tightness condition is just required "locally", and therefore the results described in this section hold in overtwisted contact $3$-manifolds if one replaces the space of smooth spheres by the space of smooth spheres with a tight neighbourhood.
\end{Remark}

The main goal of this section is Theorem \ref{thm:LinksSpheres}, which states that the space of standard embeddings of spheres into $(Y,\xi)$ fixed near the south pole is homotopy equivalent to the corresponding space of smooth embeddings. In order to prove this result we will first study the closely related space of "mini-disks". 

\begin{figure}[h!]
\centering
\includegraphics[scale=0.8]{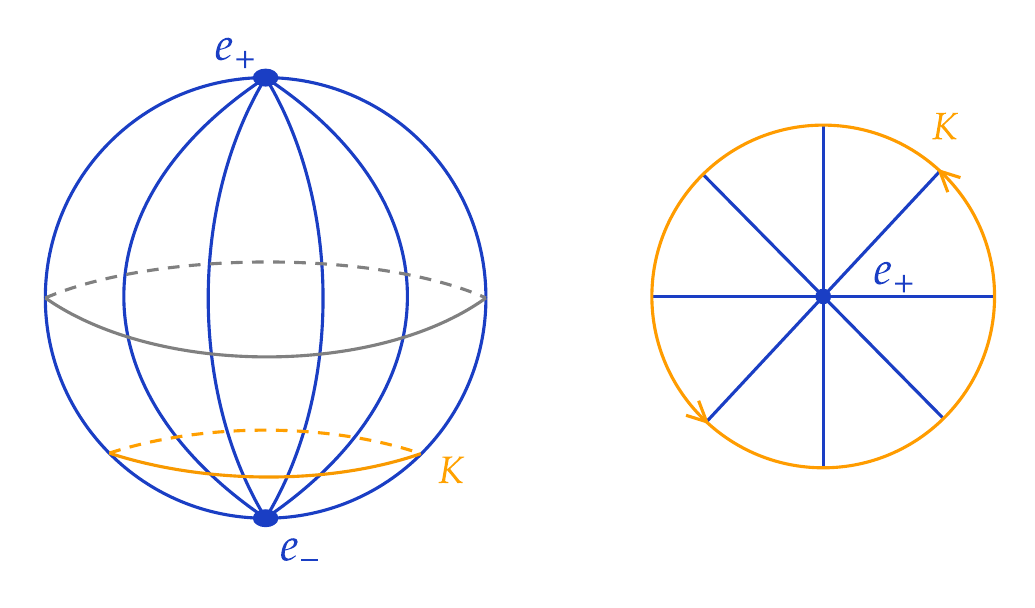}
\caption{Schematic depiction of the standard sphere and the transverse curve $K$ on the left. The mini-disk is depicted on the right.}
\label{fig:sphere}
\end{figure}

\subsection{Mini-disks in a tight $3$-manifold}

Pick a small positively transverse curve $K\subseteq e(S^2)$ surrounding the negative elliptic point $e(s)$. The curve $K$ divides the standard embedded sphere $e(S^2)$ in two disjoint disks $e(S^2)\backslash K=D^2_+\cup D^2_{-}$. Here $D^{2}_+$ contains the positive elliptic point and $D^{2}_{-}$ the negative one. In particular, we observe that the self-linking number of $K$ is $-1$. The curve $K$ is oriented as the boundary of $D^2_{+}$. Each disk $D^2_{\pm}$ is equiped with a natural parametrization induced by $e$. In particular, we will still denote by $e:\mathbb{D}^2\rightarrow (Y,\xi)$ the parametrization of $D^2_+$. A smooth embedding of a disk with positive transverse boundary with self-linking number $-1$, which is convex and induces the same characteristic foliation as $e$ is called a \textit{mini-disk}\footnote{The terminology is due to F.Presas. See Figure \ref{fig:sphere}. Observe that a mini-disk can be contracted inside small neighbourhoods of the positive elliptic point.}. 

We will denote by $\mathrm{CEmb}(\mathbb{D}^2,(Y,\xi))$ the space of embeddings of mini-disks which coincide with $e$ over an open neighbourhood of the boundary $\partial \mathbb{D}^2\subset \mathbb{D}^2$. Define the space of smooth embeddings $\mathrm{Emb}(\mathbb{D}^2,Y)$ analogously.
A consequence of Giroux's Elimination Theorem and the tightness of $(Y,\xi)$ is the following result, which will be crucial to us:

\begin{Lemma}\label{Density}(\cite{convexite,EliashbergTransverse})
Let $f\in\mathrm{Emb}(\mathbb{D}^2,Y)$ be a smooth embedding. Then there exists a $C^0$-small isotopy of $f$, relative to an open neighbourhood of the boundary, that makes makes $f$ standard.
\end{Lemma}

This result is also explained in \cite{EliashbergTransverse}. See also \cite{colin}. Here, the tightness condition is crucial. 

We will prove the following $h$-principle:
\begin{Theorem}(\cite{fmp})\label{thm:MiniDisks}
The inclusion $\mathrm{CEmb}(\mathbb{D}^2,(Y,\xi))\hookrightarrow \mathrm{Emb}(\mathbb{D}^2,Y)$ is a homotopy equivalence whenever $(Y,\xi)$ is tight.
\end{Theorem}
\begin{Remark}
\begin{itemize}
    \item The $\pi_0$-surjectivity of the previous map follows from the previous Lemma. 
    \item The $\pi_0$-injectivity and also the $\pi_1$-surjectivity follows from Colin \cite{colin}. Colin proved this by applying his Discretization Trick. However, this does not quite work parametrically due to the fact that convexity is not generic among $k$-parametric families, $k>0$.
\end{itemize}
\end{Remark}

Here, we will use the approach of \cite{fmp} based on the notion of a \textit{microfibration}, introduced by M.Gromov \cite{GromovPDR}. We will apply the following "microfibration trick", which can also be applied to an arbitrary space of convex embeddings whenever this space is dense inside the space of smooth embeeddings (Lemma \ref{Density}) and we are able to establish a corresponding local version of the $h$-principle (i.e. in a neighbourhood of a smooth embedding). These ingredients are the same as those required to effectively apply Colin's trick. Our advantage with respect to Colin's is that our techniques work parametrically. However, we lose control of the geometric picture by using an algebraic construction.

\begin{Definition}
A map $p:Y\rightarrow X$ of topological spaces is a \textit{Serre microfibration} if for every homotopy $H:\mathbb{D}^k\times [0,1]\rightarrow X$ with a lift $h_0:\mathbb{D}^k\times \{0\}\rightarrow Y$ along $p$ at time $t=0$ there exists a positive real number $\varepsilon>0$ together with an extension $h:\mathbb{D}^k \times [0,\varepsilon]\rightarrow Y$ of $h_0$ such that $p\circ h=H_{|\mathbb{D}^k\times[0,\varepsilon]}$.
\end{Definition}

A key property about microfibrations that we will use is:

\begin{Lemma}(\cite{Weiss})\label{microfibration}
Every microfibration $p:Y\rightarrow X$ with \textbf{non-empty} and contractible fiber is a Serre fibration and, therefore, a weak homotopy equivalence.
\end{Lemma}

\begin{proof}[Proof of Theorem \ref{thm:MiniDisks}]
Let $K$ be a compact parameter space and $G\subset K$ a subspace. Consider $e^k\in\mathrm{Emb}(\mathbb{D}^2,Y)$, $k\in K$, a family of smooth embeddings such that $e^k\in \mathrm{CEmb}(\mathbb{D}^2,(Y,\xi))$ for every $k\in G$. It is enough to establish the existence of a homotopy $e^k_t\in\mathrm{Emb}(\mathbb{D}^2,Y)$ such that 
\begin{itemize}
    \item $e^k_0=e^k,$
    \item $e^k_t=e^k_0$ for $k\in G$, 
    \item $e^k_1\in \mathrm{CEmb}(\mathbb{D}^2,(Y,\xi))$.
\end{itemize}
Consider any extension of the embeddings $e^k$ into a family of closed tubular neighbourhood embeddings $$E^k:\mathbb{D}^2\times [-1,1]\hookrightarrow Y$$ such that $$ E^k_{|\mathbb{D}^2\times\{0\}}=e^k.$$
Consider the space $\mathcal{B}$ of embeddings $E:\mathbb{D}^2\times [-1,1]\hookrightarrow Y$ such that $E_{|\mathbb{D}^2\times\{0\}}$ coincides with $e$ over an open neighbourhood of the boundary of $\mathbb{D}^2=\mathbb{D}^2\times\{0\}$. This space is the base of a microfibration 
$$ p:\mathcal{X}\rightarrow \mathcal{B}, (E,p_t)\mapsto E$$
where the space $\mathcal{X}$ is the space consisting of pairs $(E,p_t)$ such that 
\begin{itemize}
    \item $E\in\mathcal{B}$,
    \item $p_t\in\mathrm{Emb}(\mathbb{D}^2,E(\mathbb{D}^2\times[-1,1]))$, $t\in[0,1]$, is a homotopy of proper embeddings of disks into the closed ball $E(\mathbb{D}^2\times[-1,1])$, agreeing with the fixed embedding $e$ near the boundary, and joining $p_0=E_{|\mathbb{D}^2\times \{0\}}$ with a mini-disk embedding $p_1\in \mathrm{CEmb}(\mathbb{D}^2,(E(\mathbb{D}^2\times[-1,1]),\xi))$. 
\end{itemize}
The microfibration property is obviously satisfied. We will use Lemma \ref{microfibration} to conclude that $p$ is, in fact, a fibration. Observe that the fiber $\mathcal{F}_E$ of $p$ is non-empty because of Lemma \ref{Density}. We claim that the fiber is also contractible. This is equivalent to the fact that the space of mini-disks embeddings, fixed near the boundary, into a tight $3$-ball is homotopy equivalent to the space of  smooth embeddings, fixed near the boundary, which is a combination of Eliashberg--Mishachev's Theorem \cite{EM} and Hatcher's Theorem \cite{hatchersmale}. Indeed, denote by $(\mathbb{B}^3,\xi)=(E(\mathbb{D}^2\times[-1,1]),\xi)$ and consider any mini-disk embedding $f:\mathbb{D}^2\rightarrow (\mathbb{B}^3,\xi)$ which coincides near the boundary with $E_{|\mathbb{D}^2\times\{0\}}$. The complement of $f(\mathbb{D}^2)$ in $\mathbb{B}^3$ is given by two tight balls $\mathbb{B}^3_\pm$. Denote by $\mathrm{CEmb}(\mathbb{D}^2,(\mathbb{B}^3,\xi))$ the corresponding space of mini-disks embeddings and by $\mathrm{Emb}(\mathbb{D}^2,\mathbb{B}^3)$ the smooth analogue. There is a map between fibrations 
\[
\begin{tikzcd}
  \mathrm{Diff}(\mathbb{B}^3_+)\times \mathrm{Diff}(\mathbb{B}^3_-) \arrow{r} & \mathrm{Diff}(\mathbb{B}^3)  \arrow{r} & \mathrm{Emb}(\mathbb{D}^2,\mathbb{B}^3) \\ \mathrm{Cont}(\mathbb{B}^3_+,\xi)\times\mathrm{Cont}(\mathbb{B}^3_-,\xi)  \arrow{u}\arrow{r} &
\mathrm{Cont}(\mathbb{B}^3,\xi) \arrow{u} \arrow{r} &  \mathrm{CEmb}(\mathbb{D}^2,(\mathbb{B}^3,\xi)) \arrow{u}
\end{tikzcd}
\]
inducing homotopy equivalences of total spaces and fibers, and thus the claim follows. Then from lemma \ref{microfibration} we have that the map $p:\mathcal{X}\rightarrow \mathcal{B}$ is a Serre fibration with contractible fibers. This is enough to conclude the proof. Indeed, recall that we have built a map $j:K\rightarrow \mathcal{B}, k\mapsto E^k$, so $j^*\mathcal{X}\rightarrow K$ is a fibration with contractible non-empty fibers. We have a natural section over $G\subset K$ given by the constant homotopy: $$s:G\rightarrow \mathcal{X}, k\mapsto (E^k,p^k_t=e^k).$$ Hence, by the contractibility of the fibers we may extend this section over to $K$ and obtain $\hat{s}:K\rightarrow \mathcal{X},k\mapsto (E^k,p^k_t)$. The homotopy $$ e^k_t=p^k_t $$ solves our problem.
\end{proof}

We will need the following generalisation. Let $\mathrm{CEmb}(\sqcup_j \mathbb{D}^2,(Y,\xi))$ be the space of embeddings $e:\sqcup_j \mathbb{D}^2\hookrightarrow (Y,\xi)$ of $n$ mini-disks, all of them fixed at an open neighbourhood of $\sqcup_j \partial\mathbb{D}^2$. Denote also by $\mathrm{Emb}(\sqcup_j \mathbb{D}^2,Y)$ the corresponding space of smooth embeddings. 

\begin{Theorem}\label{MiniDisksLinks}
The natural inclusion $\mathrm{CEmb}(\sqcup_j \mathbb{D}^2,(Y,\xi))\hookrightarrow \mathrm{Emb}(\sqcup_j \mathbb{D}^2,Y) $ is a homotopy equivalence whenever $(Y,\xi)$ is tight.
\end{Theorem}
\begin{proof}
The proof follows word by word the proof of Theorem \ref{thm:MiniDisks}. In this case the microfibration built is going to have as fiber the space of isotopies of $n$ $2$-disks into $n$ disjoint tubular neighbourhoods$\cong \mathbb{D}^2\times[-1,1]$.
\end{proof}

\subsection{The space of standard spheres}

As a consequence of our previous discussion we are able to compare the homotopy types of the space of standard spheres and the space of smooth spheres in a tight contact $3$-manifold $(Y,\xi )$. For this, consider the space of smooth embeddings $\mathrm{Emb}(\sqcup_j S^2,Y)$ of $n$-disjoint spheres and the corresponding subspace of standard spheres $\mathrm{CEmb}(\sqcup_j S^2,(Y,\xi))$. Fix also an arbitrary standard embedding $e:\sqcup S^2\rightarrow (Y,\xi)$ and consider the subspaces $\mathrm{Emb}(\sqcup_j S^2,Y,\sqcup_j s_j)$ of embeddings that agree with $e$ on an open neighbourhood $\sqcup_j U_j$ of the south pole $s_j$ of each sphere. Here, we assume that the boundary $e_{|\sqcup_j \partial U_j}$ parametrizes $n$ disjoint positively transverse knots $K_j$ as in the previous section. Similarly, consider the analogous subspace of standard embeddings $\mathrm{CEmb}(\sqcup_j S^2,(Y,\xi),\sqcup_j s_j)$. Observe the following: the space $\mathrm{CEmb}(\sqcup_j S^2,(Y,\xi),\sqcup_j s_j)$ is homotopy equivalent to the space of $n$ mini-disks embeddings into the tight contact manifold with convex boundary obtained from $(Y,\xi)$ by removing an open neighbourhood of $e(\sqcup_j U_j)$ whose boundary parametrizes $K_j$. The same observation applies to the space $\mathrm{Emb}(\sqcup_j S^2,Y,\sqcup_j s_j)$. We obtain:

\begin{Theorem}\label{thm:LinksSpheres}
Assume that $(Y,\xi)$ is tight. Then 
\begin{itemize}
    \item The inclusion $\mathrm{CEmb}(\sqcup_j S^2,(Y,\xi),\sqcup_j s)\hookrightarrow\mathrm{Emb}(\sqcup_j S^2,Y,\sqcup_j s_j)$ is a homotopy equivalence.
    \item For every $k\geq 1$ the natural homomorphism
    $$ \pi_k(\mathrm{SO}(3)^n,\mathrm{U}(1)^n)\rightarrow \pi_k(\mathrm{Emb}(\sqcup_j S^2,Y),\mathrm{CEmb}(\sqcup_j S^2,(Y,\xi)))$$ induced by reparametrization on the source is an isomorphism.
\end{itemize}
\end{Theorem}
\begin{proof}
As explained above, the proof of the first assertion follows from Theorem \ref{MiniDisksLinks}. For the second assertion note that there is a natural map of fibrations given by the evaluation at the $n$ south poles: 

\[
\begin{tikzcd}
  \mathrm{Emb}(\sqcup_j S^2,Y,\sqcup_j s_j) \arrow{r} & \mathrm{Emb}(\sqcup_j S^2,Y)  \arrow{r} & \mathrm{Fr}_n(Y)\\ \mathrm{CEmb}(\sqcup_j S^2,(Y,\xi),\sqcup_j s_j)  \arrow{u}\arrow{r} &
\mathrm{CEmb}(\sqcup_j S^2,(Y,\xi)) \arrow{u} \arrow{r} &  \mathrm{CFr}_n(Y,\xi) \arrow{u}
\end{tikzcd}
\]
in which the vertical maps are inclusions. Here, the base $\mathrm{Fr}_n (Y)$ is the space of framings over $n$ different points of $M$, that is, the total space of a fiber bundle over the configuration space $\mathrm{Conf}_n(Y)$ with fiber $\approx \mathrm{GL}^+(3)^n$, and likewise for $\mathrm{CFr}_n (Y,\xi)$ but with contact frames. Observe that the map between the fibers is a homotopy equivalence because of the first assertion, so that the homomorphism induced by the evaluation map
\begin{align*}
    \pi_k(\mathrm{Emb}(\sqcup_j S^2,Y),\mathrm{CEmb}(\sqcup_j S^2,(Y,\xi))) & \rightarrow  \pi_k (\mathrm{Fr}_n (Y),\mathrm{CFr}_n (Y,\xi)) \\ 
    & \cong\pi_k(\mathrm{SO}(3)^n,\mathrm{U}(1)^n)
\end{align*} 
is an isomorphism and defines an inverse to the reparametrization map. This concludes the proof. \end{proof}

\subsection{Standard spheres in sums of two irreducible $3$-manifolds}\label{embSO3}
In this section we establish Theorem \ref{embthm}. We first discuss its smooth counterpart. The relevant reference on this topic is Hatcher's work \cite{S1xS2}. Let $Y_\# = Y_- \# Y_+$ with $Y_{\pm}$ now \textit{irreducible}. Let $\mathrm{Emb} (S^2 , Y_\# )_{S_\#} \subset \mathrm{Emb}(S^2 , Y_\# )$ be the subspace of smooth co-oriented embeddings $S^2 \hookrightarrow{Y_\#}$ isotopic to a fixed given one $S_\#$, and let $$\mathcal{S} = \mathrm{Emb}(S^2 , Y_\# )_{S_\#} /\mathrm{Diff}(S^2)$$ be the space of \textit{unparametrised} co-oriented non-trivial spheres. Hatcher \cite{S1xS2} proved that $\mathcal{S}$ is contractible. We also have a fibration 
\begin{align*}
\mathrm{SO}(3) \simeq \mathrm{Diff}(S^2 ) \rightarrow \mathrm{Emb} (S^2 , Y_\# )_{S_\#} \rightarrow \mathcal{S}
\end{align*}
and hence $$\mathrm{Emb} (S^2 , Y_\# )_{S_\#} \simeq \mathrm{SO}(3).$$ 

\begin{proof}[Proof of Theorem \ref{embthm}] Immediate from the long exact sequence of pairs associated to the horizontal maps in the following commutative diagram 
\[
\begin{tikzcd}
\mathrm{CEmb}\big(S^2 , (Y_\#, \xi_\# ) \big)_{S_\#} \arrow{r} & \mathrm{Emb}\big(S^2 , Y_\# \big)_{S_\#}\\
\mathrm{U}(1) \arrow{r} \arrow{u} & \mathrm{SO}(3) \arrow{u}{\simeq}
\end{tikzcd}
\]
combined with Theorem \ref{thm:LinksSpheres}.
\end{proof}

\section{Families of contact structures on sums of contact $3$-manifolds} \label{connectedsumsection}

In this section we establish the main results of the article, Theorems \ref{mainthm}, \ref{mainthm2} and \ref{sumcontact2} by combining the tools discussed in \S \ref{monopolesection} and \S \ref{spheressection}.

\subsection{The space of tight contact structures on a sum}

Consider $n+1$ tight contact $3$-manifolds $(Y_j , \xi_j )$, $j = 0 , \ldots , n$ with $n\geq 1$. Let $(Y_\#, \xi_\# )$ be their connected sum, which we build as follows. We fix Darboux balls $B_{0-} \subset Y_0$, $B_{n+} \subset Y_n$ and for each $0 < j < n$ we fix two disjoint Darboux balls $B_{j\pm} \subset Y_j$. Then the connected sum $(Y_\#,\xi_\#)$ is formed by gluing in the following order
$$ \big( Y_0 \setminus B_{0-} \big) \bigcup_{\partial B_{0-} = -\partial B_{1+}} \big( Y_1 \setminus ( B_{1+} \cup B_{1-} ) \big) \cdots 
\bigcup_{\partial B_{(n-1)-} =- \partial B_{n+}} \big( Y_n \setminus B_{n+} \big). $$
We will denote by $e_j:S^2\hookrightarrow (Y_\#,\xi_\#)$, $j = 1,\ldots,n$, the embedding of the $j^{\mathrm{th}}$ separating standard sphere given by $\partial B_{(j-1)-}=-\partial B_{j+}$ in the connected sum $(Y_\#,\xi_\#)$. Denote by $s_j$ the south pole on the $j^{\mathrm{th}}$ sphere, regarded as a point in $e_j (S^2 ) \subset Y_\#$.

We will denote by $\mathrm{Tight}(Y,B)$ the space of tight contact structures on $Y$ that are fixed on a Darboux ball $B$ and by $\mathrm{Tight}(Y,B,B^{\prime})$ the subspace of $\mathrm{Tight}(Y,B)$ given by contact structures that are fixed on a second Darboux ball $B^\prime$ disjoint from $B$. A classical result of Colin \cite{colin} asserts that the contact manifold $(Y_\#,\xi_\#)$ is tight, and we have a well-defined map
\begin{equation}\label{eq:ConnectedSum}
\#_{n+1}: \mathrm{Tight}(Y_0,B_{0-})\times \prod_{j=1}^{n-1} \mathrm{Tight}(Y_j,B_{j+},B_{j-})\times \mathrm{Tight}(Y_n,B_{n+})\rightarrow \mathrm{Tight}(Y_\#).
\end{equation}

On the other hand, the evaluation map of each tight contact structure on $Y$ at the south poles $s_j$ defines a fibration 
\begin{equation}\label{eq:GeneralEvaluation}
 ev_{n+1}:\mathrm{Tight}(Y_\#)\rightarrow (S^2)^{n}.
\end{equation}

The fiber $\mathcal{F}$ of $ev_{n+1}$ over $(\xi_\# (s_j ) )$ has the homotopy type of the space of tight contact structures on $Y_\#$ that agree with $\xi_\#$ over $n$ disjoint Darboux balls $B_{\#j}$ around $s_j$. Therefore, there is a natural inclusion 
$$ i_\#:\mathrm{Tight}(Y_0,B_{0-})\times \prod_{j=1}^{n-1} \mathrm{Tight}(Y_j,B_{j+},B_{j-})\times \mathrm{Tight}(Y_n,B_{n+})\hookrightarrow \mathcal{F}. $$
We establish the following stronger version of Theorem \ref{sumcontact2}:

\begin{Theorem}\label{thm:ConnectedSumTight}
The inclusion $i_\#$ is a homotopy equivalence.
\end{Theorem}
\begin{Remark}
    Since $S^2$ is simply connected we deduce from the long exact sequence in homotopy groups of (\ref{eq:GeneralEvaluation}) that 
    $$ \pi_0 \big( \mathrm{Tight}(Y_\#) \big) \cong \prod_{j=0}^{n} \pi_0\big( \mathrm{Tight}(Y_j)\big) $$
    which is the classical result of Colin \cite{colin}.  
    
\end{Remark}
\begin{proof}
Let $K$ be a compact parameter space and $G\subseteq K$ a subspace. It is enough to prove that: if $\xi^k\in \mathcal{F}$ is a $K$-family of tight contact structures on $Y_\#$ that coincide with $\xi_\#$ over the $n$ Darboux balls $B_{\#j }$ and such that $\xi^k\in \mathrm{Im} ( i_\# )$ for $k\in G$, then there exists a homotopy of tight contact structures $\xi^k_t$, $t\in[0,1]$, such that 
\begin{itemize}
    \item $\xi^k_0=\xi^k$,
    \item $\xi^k_t=\xi^k$ for $k\in G$ and 
    \item $\xi^k_1\in \mathrm{Im}(i_\# ).$
\end{itemize}
The key point is to observe that $\xi^k\in \mathrm{Im}(i_\# )$ if and only if the embeddings $e_j:S^2\hookrightarrow (Y_\#,\xi^k)$ are standard for $j = 1 , \ldots, n $. For a given tight contact structure $\xi$ denote by $$\mathcal{CE}_{\xi} := \mathrm{CEmb}(\sqcup_{j=1}^{n} S^2,(Y_\#,\xi),\sqcup_{j=1}^{n} s_j))$$ the space of standard embeddings of $n$ disjoint spheres that coincide with $(e_j )$ over a neighbourhood of the south poles $(s_j )$, and by $$\mathcal{E} := \mathrm{Emb}(\sqcup_{j=1}^{n} S^2,Y_\#,\sqcup_{j=1}^{n} s_j))$$ the analogous space of smooth embeddings. Consider the space $\mathcal{X}$ of pairs $(\xi,e_t)$ where $\xi\in\mathcal{F}$ and $e_t\in\mathcal{E}$, with $t\in[0,1]$, is a homotopy of embeddings with $e_0=e$ and $e_1\in\mathcal{CE}_\xi$. There is a natural forgetful map 
$$ p:\mathcal{X}\rightarrow \mathcal{F}, (\xi,e_t)\mapsto \xi, $$
which is in fact a fibration because of Lemma \ref{lem:Gray}. By Theorem \ref{thm:LinksSpheres} we know that the inclusion $ \mathcal{CE}_{\xi} \rightarrow \mathcal{E}$ is a homotopy equivalence. Therefore, the fibers of the previous fibration are contractible.

This is enough to conclude the proof. Indeed, our initial family $\xi^k$ is given by a map $j:K\rightarrow \mathcal{F}$ and the pullback fibration $j^*\mathcal{X}\rightarrow K$ has a well-defined section over $G\subseteq K$ given by the constant isotopy $e^{k}_{t}=e$, $(k,t)\in G\times [0,1]$. Since the fiber of this fibration is contractible we can extend this section over $K$ obtaining a section $e^{k}_{t}$, $(k,t)\in K\times [0,1]$. Then we apply the smooth isotopy extension theorem to this family of embeddings to find an isotopy $ \varphi^{k}_{t}\in \mathrm{Diff}(Y_\#)$, $(k,t)\in K\times [0,1]$, such that 
\begin{itemize}
    \item $\varphi^{k}_{0}=\mathrm{Id}$,
    \item $\varphi^{k}_{t}$ is the identity over a neighbourhood of the south poles $(s_j )$,
    \item $\varphi^{k}_{t}\circ e=e^{k}_{t}$, 
    \item $\varphi^{k}_{t}=\mathrm{Id}$ for $(k,t)\in G\times [0,1]$. 
\end{itemize}
The homotopy of contact structures $\xi^k_t=(\varphi^{k}_{t})^*\xi^k$ solves the problem since now $e=(\varphi^{k}_{1})^{-1}\circ e^{k}_{1}$ is standard for $(\varphi^{k}_{t})^* \xi^k$ because $e^{k}_{1}$ is standard for $\xi^k$.
\end{proof}

\subsection{Diffeomorphisms of connected sums of two irreducible $3$-manifolds} 

Consider $Y_\# = Y_- \# Y_+$ with $Y_{\pm}$ \textit{irreducible}. Recall from \S \ref{embSO3} that Hatcher \cite{S1xS2} proved $$\mathrm{Emb} (S^2 , Y_\# )_{S_\#} \simeq \mathrm{SO}(3).$$ This has the following useful consequence:

\begin{Lemma}\label{pi1diff} Suppose that $Y_{\pm}$ are aspherical (i.e. irreducible and with infinite fundamental group). Then $\pi_1 \mathrm{Diff} (Y_\# ) = 0.$
\end{Lemma}
\begin{proof}[Proof of Lemma \ref{pi1diff}]

From the fibration (\ref{diffsphere}) we have an exact sequence

\begin{tikzpicture}[baseline= (a).base]
\node[scale=1 , trim left=-6cm] (a) at (0,0){
\begin{tikzcd}
      \pi_1 \mathrm{Diff}(Y_- , B_- ) \times \pi_1 \mathrm{Diff}(Y_+ , B_+ ) \arrow{r} &  \pi_1 \mathrm{Diff}(Y_\# ) \arrow{r} &  \mathbb{Z}_2 \arrow[out=0, in=180]{dll} \\
      \pi_0 \mathrm{Diff}(Y_- , B_- ) \times \pi_0 \mathrm{Diff}(Y_+ , B_+ )  &  & 
\end{tikzcd}
};
\end{tikzpicture}
Under the connecting map, the non-trivial element in $ \mathbb{Z}/2$ maps to $\tau_{\partial B_-} \tau_{\partial B_+ } \in \pi_0 \mathrm{Diff} ( Y_- , B_- ) \times \pi_0 \mathrm{Diff}(Y_+ , B_+ )$. We saw in the proof of Corollary \ref{criterionB} that the Dehn twists $\tau_{\partial B_\pm} \in \pi_0 \mathrm{Diff}(Y_\pm , B_\pm )$ are non-trivial and $\pi_1 \mathrm{Diff}(Y_\pm , B_\pm ) = 0$. From this and the exact sequence above it now follows that $\pi_1 \mathrm{Diff}(Y_\# ) = 0$. \end{proof}

\subsection{Proof of Theorem \ref{mainthm}} As we've been doing so far, all homologies considered below are taken with $\mathbb{Q}$ coefficients, unless otherwise noted.

By Theorem \ref{thm:ConnectedSumTight} we have $$ \mathcal{C}(Y_\# , \xi_\# , B_\# ) \simeq \mathcal{C}(Y_- , \xi_- , B_- ) \times \mathcal{C}(Y_+ , \xi_+ , B_+ )$$ and then by Proposition \ref{obstrsum} the obstruction class $\mathcal{O}_{\xi_\#} \in \pi_1 \mathcal{C}(Y_\# , \xi_\# , B_\# )$ to finding a homotopy section of $ev_{\#} : \mathcal{C}(Y_\# , \xi_\# ) \rightarrow S^2$ corresponds to $$\mathcal{O}_{\xi_\#} \cong ( \mathcal{O}_{\xi_-}  , \mathcal{O}_{\xi_+ } ) \in \pi_1 \mathcal{C}(Y_- , \xi_- , B_- ) \times \pi_1 \mathcal{C}(Y_+ , \xi_+ , B_+ ) .$$
We recall that all homologies are taken with $\mathbb{Q}$ coefficients. A portion of the Wang long exact sequence for the fibration $ev_{B_\#}$ is 
\begin{align*}
     \mathbb{Q} \xrightarrow{\delta} \mathrm{H}_1 \big(  \mathcal{C}(Y_- , \xi_- , B_- )\big) \oplus \mathrm{H}_1 \big( \mathcal{C}(Y_+ , \xi_+ , B_+ ) \big) \rightarrow \mathrm{H}_1 \big( \mathcal{C}(Y_\# , \xi_\# ) \big) \rightarrow 0
\end{align*}
where $\delta (1 ) = \mathcal{O}_{\xi_\#} = ( \mathcal{O}_{\xi_-}  , \mathcal{O}_{\xi_+ } )$. In the latter formula, and in what follows, we will incur in a small abuse of notation by denoting the obstruction class and its image in homology by the same symbol.

The non-trivial input from Floer theory appears now. Because $\mathbf{c}(\xi_\pm ) \notin \mathrm{Im}U$ then by the Corollary \ref{cor:NoLagrangianRot} to Theorem \ref{commdiag} the classes $\mathcal{O}_{\xi_{\pm}}$ are non-trivial in $\mathrm{H}_1 \big( \mathcal{C}(Y_\pm , \xi , B_\pm) \big)$. It then follows from the Wang exact sequence that the class $(\mathcal{O}_{\xi_-} , 0 ) $ is not in the image of $\delta$, thus the image of $(\mathcal{O}_{\xi_-} , 0 )$ in $\mathrm{H}_1 \big(  \mathcal{C}(Y_\# , \xi_\# ) \big) $ is non-trivial.

Now, from Lemma \ref{pi1diff} we have $\pi_1 \mathrm{Diff}(Y_\# ) = 0$. Note that the hypothesis of that Lemma indeed apply, because the manifolds $Y_\pm $ are aspherical. To see this, recall that an irreducible $3$-manifold is aspherical precisely when it is not one of the quotients $M_\Gamma$ of $S^3$ by a finite subgroup $\Gamma$ of $SO(4)$. The manifolds $M_\Gamma $ have $\widecheck{\mathrm{HM}} ( - M_\Gamma , \mathfrak{s} ) = \mathbb{Q}[U, U^{-1}]$ in every spin-c structure $\mathfrak{s}$ because $M_\Gamma$ is a rational homology sphere with a positive scalar curvature metric (see Proposition 36.1.3 in \cite{KM}). In particular, the $U$ map is surjective on $\widecheck{\mathrm{HM}}(- M_\Gamma , \mathfrak{s} )$. The manifolds $Y_\pm$ are irreducible but can't be of the form $M_\Gamma$ since $\mathbf{c}(\xi_\pm) \notin \mathrm{Im}U$.

Then, by the long exact sequence in homotopy groups of (\ref{moser}) it follows that 
$$\mathrm{H}_1 \big( \mathcal{C}(Y_\#, \xi_\# ) ; \mathbb{Z} \big) \cong \mathrm{Ab}\Big(  \pi_0 \mathrm{Cont}_0 (Y_\#, \xi_\# ) \Big) .$$ Under this isomorphism, the non-trivial class $(\mathcal{O}_{\xi_-} , 0)$ corresponds to the class of the squared Dehn twist $\tau_{S_\#}^2$ by Proposition \ref{moserdehn}. This proves that $\tau_{S_\#}^2$ is not contact isotopic to the identity. Since we've shown that $(\mathcal{O}_{\xi_-} , 0 )$ is non-trivial \textit{rationally}, it follows that all the even powers of $\tau_{S_\#}$ (and therefore all the powers) are also not contact isotopic to the identity. This completes the proof of Theorem \ref{mainthm}(A).

We now establish Theorem \ref{mainthm}(B). By Lemma \ref{ftrivialtwist} we have that the image of $\tau_{\partial B_\pm }$ in $ \pi_0 \mathrm{FCont}_0 (Y_\pm, \xi_\pm , B_\pm )$ is trivial. Hence, so is the image of $\tau_{S_\#}^2$ in $\pi_0 \mathrm{FCont}_0 (Y_\# , \xi_\# )$. The proof of Theorem \ref{mainthm} is now complete. $\square$

\begin{Remark}
    Working with $\mathbb{Z}$ coefficients rather than $\mathbb{Q}$, we can establish the following analogue of Theorem \ref{mainthm}(A) by the same argument. If $2 \mathbf{c}(\xi_\pm ; \mathbb{Z} ) \neq 0$ in $\widecheck{\mathrm{HM}} (-Y_\pm )$ and $0 \neq k \in \mathbb{Z}$ satisfies $k  \mathbf{c}(\xi_\pm ; \mathbb{Z} ) \notin \mathrm{Im} U$, then the $2k$-fold iterate $\tau_{S_\#}^{2k}$ is not contact isotopic to the identity. All examples known to the authors where the latter hypothesis is satisfied for some $k \neq 0$ also satisfy the stronger $\mathbb{Q}$ version of the hypothesis. The assumption $2 \mathbf{c}(\xi_\pm ; \mathbb{Z} ) \neq 0$ guarantees the orientability of the moduli spaces involved in the construction of the families contact invariant (see Corollary 1.8 in \cite{yo}).
    \end{Remark}

\subsection{Proof of Theorem \ref{mainthm2}}

We write $(Y, \xi ) = (Y_0 , \xi_0 ) \# \cdots \# (Y_n , \xi_n ) \# (Y_{n+1}, \xi_{n+1} )$ where $(Y_0 , \xi_1 ), \ldots , (Y_n , \xi_n )$ are those prime summands of $(Y, \xi )$ such that $\mathbf{c}(\xi_j ) \notin \mathrm{Im}U$ and the Euler class of $\xi_j$ vanishes, and $(Y_{n+1}, \xi_{n+1} )$ is the sum of the remaining prime summands. We take the latter to be $(S^3 , \xi_{\mathrm{st}} )$ if there are no prime summands remaining. We choose Darboux balls $B_{0-}\subset Y_0$, $B_{n+1,+} \subset Y_{n+1}$ and for $j = 1, \ldots , n$ we choose two Darboux balls $B_{j\pm} \subset Y_j$ disjoint from each other. We may take the connected sum $(Y, \xi )$ to be built by gluing in the following order
$$  \big( Y_0 \setminus B_{0-}  \big) \bigcup_{\partial B_{0-} = - \partial B_{1+}}  \big( Y_1 \setminus (B_{1+} \cup B_{1-} ) \big) \cdots  \big( Y_n \setminus (B_{n+} \cup B_{n-} ) \big) \bigcup_{\partial B_{n-} = - \partial B_{n+1,+}} \big( Y_{n+1} \setminus B_{n+1,+} \big)$$
with $n+1$ separating spheres. We fix $n+1$ Darboux balls $B_{\#j}$, $j = 1, \ldots , n+1$, centered at the south poles of the separating spheres (i.e. $B_{\#j}$ is centered at the south pole of the sphere which separates the pieces $Y_{j-1} \setminus B_{j-1,-}$ and $Y_{j+} \setminus B_{j+}$) and which are disjoint from each other.

Consider the evaluation map at the $n+1$ south poles of the spheres, which provides a fibration
\begin{align}
\mathcal{F} \rightarrow \mathcal{C}(Y, \xi ) \rightarrow (S^2 )^{n+1} . \label{evmanypoints}
\end{align}
Theorem \ref{thm:ConnectedSumTight} identifies the fiber as
\begin{align*}
    \mathcal{F} \simeq \mathcal{C}(Y_0 , B_{0-} ) \times \Big( \prod_{j = 1, \ldots , n+1} \mathcal{C}(Y_j , B_{j+} \cup B_{j-} ) \Big) \times \mathcal{C}(Y_{n+1} , B_{n+1,+} ).
\end{align*}
Recall that we have a homotopy equivalence  
$$ 
\mathcal{C}(Y_j , B_{j+} \cup B_{j-} ) \simeq  \Omega S^2  \times \mathcal{C} (Y_j , B_{j-} ) \, ,
$$
since the evaluation map $ev_{B_{j+}} : \mathcal{C}(Y_j , B_{j-} ) \rightarrow S^2 $ is null-homotopic (a null-homotopy is obtained by dragging the evaluation point from $B_{j+}$ into $B_{j-}$, and this yields the required homotopy equivalence). Thus, we have
\[
\mathcal{F} \simeq \mathcal{C}(Y_0 , B_{0-}) \times \Omega S^2 \times \mathcal{C} (Y_1 , B_{1-} ) \times \cdots \times \Omega S^2 \times \mathcal{C}(Y_n , B_{n-} ) \times \mathcal{C}(Y_{n+1} , B_{n+1,+}).
\]

The connecting map in the long exact sequence in homotopy groups of the fibration (\ref{evmanypoints}) yields a homomorphism
\[
\delta : \mathbb{Z}^{n+1} \rightarrow \pi_1 \mathcal{C}(Y_0 , B_{0-}) \times \mathbb{Z} \times \pi_1 \mathcal{C} (Y_1 , B_{1-} ) \times \cdots \times \mathbb{Z} \times \pi_1 \mathcal{C}(Y_n , B_{n-} ) \times \pi_1 \mathcal{C}(Y_{n+1} , B_{n+1,+}). 
\]
which we now calculate. 

\begin{Lemma}\label{Lemma:obstrsum2} For $(a_1 , \ldots , a_{n+1} ) \in \mathbb{Z}^{n+1}$ we have
\[\delta (a_1 , \ldots , a_{n+1} ) = \big( \, a_1 \cdot \mathcal{O}_{\xi_0} \, , \, a_1 \, , a_2 \cdot \mathcal{O}_{\xi_1} \, , \cdots , \, a_n \, , a_{n+1} \cdot \mathcal{O}_{\xi_n} \, , \, a_{n+1} \cdot \mathcal{O}_{\xi_{n+1}} \, \big) \, .\]
\end{Lemma}

\begin{proof}
The argument we use is modelled on the proof of Proposition \ref{obstrsum}. It suffices to work in the local model where $(Y_j , \xi_j ) = (\mathbb{B}^3 , \xi_\mathrm{st} )$ for all $j = 0 , \ldots , n+1$. We choose $2n+2$ paths $\gamma_{0-}$, $\gamma_{1\pm}$, $\ldots \gamma_{n \pm}$, $\gamma_{n+1,+}$ in $Y$, where each $\gamma_{j+}$ goes from $B_{\# j}$ to $\partial Y_j \subset \partial Y$, and each $\gamma_{j-}$ goes from $B_{\# j+1}$ to $\partial Y_j \subset \partial Y$. We consider the following commutative diagram of maps and spaces
\[
\begin{tikzcd}
\mathcal{C}(Y , \cup_{j = 1}^{n+1} B_{\# j} ) \arrow{r} \arrow{d}{ev_{\gamma } } &  \mathcal{C}(Y  ) \arrow{d}{ev_\gamma} \arrow{r}{ev_{B_\#}}  & (S^2)^{n+1} \arrow{d}{ \Delta^{n+1}  } \\ 
(\Omega S^2)^{2n+2}    \arrow{r} & (PS^2)^{2n+2} \arrow{r} & (S^2 )^{2n+2} \, .
\end{tikzcd}
\]
Here $ev_{\#B}$ stands for the evaluation map at the centers of the $n+1$ balls $B_{\#j}$, and the bottom row is given by $2n+2$ product of the path fibration on $S^2$. In particular, both rows are fibration sequences. The maps denoted $ev_\gamma$ stand for evaluation of contact structures along the $2n+2$ paths chosen above, and $\Delta : S^2 \rightarrow (S^2)^2$ is the diagonal map.

Each of the two fibrations $ev_{B_{0-}} : \mathcal{C}(Y_0 ) \rightarrow S^2$, $ev_{B_{n+1,+}} : \mathcal{C}(Y_{n+1} ) \rightarrow S^2$ is identified with the path fibration on $S^2$, by Lemma \ref{Lemma:Olocal}. Similarly, each of the $n$ fibrations $ev_{B_{j+}} \times ev_{B_{j-}} : \mathcal{C}(Y_j ) \rightarrow (S^2 )^2$ with $j = 1, \ldots , n$ is identified with two copies of the path fibration on $S^2$. Using these, we identify the bottom row of the first diagram with the product of these $n+2$ fibrations.

The left-most vertical map in the first diagram is a homotopy equivalence, which follows by an argument similar to the proof of Lemma \ref{Lemma:Olocal}. Consider the inclusion map 
\[
j : \mathcal{C}(Y_0 , B_{0-} ) \times \big( \prod_{j = 1}^{n} \mathcal{C}(Y_j , B_{j+} \cup B_{j-} ) \big) \times \mathcal{C}(Y_{n+1} , B_{n+1,+} ) \rightarrow \mathcal{C}(Y , \cup_{j = 1}^{n+1} B_{\# j} ).
\]
Under the identification of the bottom row of the first diagram with the product of the $n+2$ fibrations from the previous paragraph, the map $j$ becomes the homotopy inverse of the left-most vertical map in the first diagram, as in the proof of Lemma \ref{obstrsum}. The required result follows now from the commuting square obtained from taking homotopy groups in the first diagram:
\[
\begin{tikzcd}
(\pi_2 S^2 )^{n+1} \arrow{r} \arrow{d}{\Delta^{n+1}} & \pi_1 \mathcal{C}(Y, \bigcup_{j = 1}^{n+1} B_{\# j} ) \\
(\pi_2 S^2 )^{2n+2} \arrow{r}  & \pi_1 \mathcal{C}(Y_0 , B_{0-} ) \times \big( \prod_{j = 1}^{n} \pi_1 \mathcal{C}(Y_j , B_{j+} \cup B_{j-} ) \big) \times \pi_1 \mathcal{C}(Y_{n+1} , B_{n+1,+} ) \arrow{u}{j}.
 \end{tikzcd}
\]
\end{proof}

With this in place, we now look at the Serre spectral sequence of the fibration (\ref{evmanypoints}). From it we can assemble an exact sequence
\begin{align*}
    \mathbb{Q}^{n+1} \xrightarrow{\delta} \mathrm{H}_1 (\mathcal{F} ) \rightarrow \mathrm{H}_1 \big(\mathcal{C}(Y, \xi ) \big) \rightarrow 0 
\end{align*}
where $\delta$ is given by the same formula as in Lemma \ref{Lemma:obstrsum2}. By $\mathbf{c}(\xi_j ) \notin \mathrm{Im}U$ and the Corollary \ref{cor:NoLagrangianRot} to Theorem \ref{commdiag} we again deduce that the classes $\mathcal{O}_{\xi_j}$ ($j = 0 , \ldots , n$) are homologically non-trivial (over $\mathbb{Q}$). Hence the $n$-dimensional subspace of $\mathrm{H}_1 \big( \mathcal{F} \big)$ given by the elements
\begin{align*}
    \big( \, b_1 \cdot \mathcal{O}_{\xi_0} \, , \, 0 \, , \, b_2 \cdot \mathcal{O}_{\xi_1} \, , \, 0 \, , \, \ldots \, , 0  \, , \, b_{n} \cdot \mathcal{O}_{\xi_{n-1}} \, , \, 0 \,  , \, 0  \, , \, 0 \, \big) \quad , \quad (b_j ) \in \mathbb{Q}^{n}
\end{align*}
injects as a subspace of $\mathrm{H}_1 \big( \mathcal{C}(Y, \xi ) \big) $. The proof of the formal triviality assertion is similar to the one given for Theorem \ref{mainthm}. The proof of Theorem \ref{mainthm2} is now complete. $\square$

\begin{Remark}\label{why}
When $Y $ is the sum of two aspherical $3$-manifolds we have $\pi_1 \mathrm{Diff}(Y) = 0$ (see Lemma \ref{pi1diff}). In the proof of Theorem \ref{mainthm} this allowed us to pass from a non-trival element in $\pi_1 \mathcal{C}(Y, \xi )$ to a non-trivial element in $\pi_0 \mathrm{Cont}_0 (Y, \xi )$ via the fibration (\ref{moser}). This is a special situation. For instance, if $Y$ is instead the sum of \textit{at least three} aspherical $3$-manifolds then it is known that $\pi_1 \mathrm{Diff} (Y)$ is not finitely generated \cite{diffsum3}. A better control on $\pi_1 \mathrm{Diff}(Y)$ for general $Y$ would allow us to understand whether the exotic loops of contact structures that we find in Theorem \ref{mainthm2} yield non-trivial contactomorphisms.
\end{Remark}

\section{Exotic phenomena in overtwisted contact $3$-manifolds} \label{OTsection}


In this final section we exhibit examples of $1$-parametric exotic phenomena in \textit{overtwisted} contact $3$-manifolds. 

On a heuristic level, Eliashberg's overtwisted $h$-principle \cite{EliashbergOT} is based on applying Gromov's \textit{h}-principle for open manifolds to the complement of a $3$-ball and using the overtwisted disk to fill in the ball. In the same spirit of this idea is what we call the "overtwisted escape principle", explained to us by F. Presas, which is a general strategy for proving an $h$-principle for a family of objects in a contact manifold $(Y, \xi )$. First, perform the connected sum with an overtwisted manifold $(M, \xi_{\mathrm{ot}})$, in order to apply the overtwisted \textit{h}-principle \cite{EliashbergOT,BEM} in the contact $3$-manifold $(Y,\xi)\#(M,\xi_{\mathrm{ot}})$. This could be thought of as analogous to opening up the $3$-manifold in the previous situation. Secondly, try to isotope the objects for which you want an $h$-principle so that they avoid ("escape") the overtwisted region $(M,\xi_{\mathrm{ot}})\setminus B$, where $B$ is a Darboux ball. However, there could be obstructions to carrying out this second step. There are two scenarios: if these obstructions can be sorted out then our initial problem satisfies an $h$-principle; if not these obstructions should give rise to an exotic phenomenon in the overtwisted contact manifold $(Y,\xi)\#(M,\xi_{\mathrm{ot}})$.  In \cite{CPP} the authors succesfully carry out this procedure to prove an existence \textit{h}-principle for codimension $2$ isocontact embeddings. Next, we will instead start out of a problem in $(Y, \xi )$ which we know is geometrically obstructed a priori, and from this deduce an exotic overtwisted phenomenon.

Let $e:S^2\rightarrow (Y,\xi)$ be a standard embedding into a contact manifold $(Y,\xi)$. A \em formal standard embedding \em of a sphere into $(Y,\xi)$ is a pair $(f,F^s)$, $s\in[0,1]$, such that $f\in\mathrm{Emb}(S^2,Y)$ is a smooth embedding and $F^s:TS^2\rightarrow f^*TY$ is a homotopy of vector bundle injections with $F^0 = df$ and $(F^1)^*\xi=e^*\xi\subset TS^2$. 
We will denote by $\mathrm{FCEmb}(S^2,(Y,\xi))$ the space of formal standard embeddings and by $\mathrm{FCEmb}(S^2,(Y,\xi),s)$ the subspace of formal standard embedding that coincide with $e$ over an open neighbourhood $U$ of the south pole $s\in S^2$.

Let $(M,\xi_\mathrm{ot})$ be an overtwisted contact $3$-manifold. Consider the overtwisted contact $3$-manifold $(Y_\#,\xi_\#)=(Y,\xi)\#(M,\xi_\mathrm{ot})$. We will consider the spaces $\mathrm{CEmb}(S^2,(Y_\#,\xi_\#),s)$ and $\mathrm{FCEmb}( S^2,(Y_\#,\xi_\#),s)$ as pointed spaces with base point given by the separating sphere $e : S^2 \hookrightarrow (Y_\#, \xi_\# )$. We have a natural inclusion $\mathrm{CEmb}(S^2,(Y_\#,\xi_\#),s)\hookrightarrow \mathrm{FCEmb}(S^2,(Y_\#,\xi_\#),s)$. From our previous discussion and the theory developed in this article we deduce the following

\begin{Corollary}\label{cor:RigidityInOTSpheres}
Assume that $(Y,\xi)$ is irreducible, $\xi$ has vanishing Euler class and $\mathbf{c}(\xi) \notin \mathrm{Im} U$. Then, there exists an element with infinite order in
$$ \mathrm{Ker} \Big(  \pi_1 \mathrm{CEmb}(S^2,(Y_\#,\xi_\#),s)\rightarrow \pi_1\mathrm{FCEmb}(S^2,(Y_\#,\xi_\#),s) \Big).$$
\end{Corollary}
\begin{Remark}
\begin{itemize}
\item This should be compared with Theorem \ref{thm:LinksSpheres}, which in particular asserts that this type of phenomenon does not happen when the underlying contact manifold is tight. 
\item Under the same assumptions, our proof also yields an element with infinite order in
$$ \mathrm{Ker} \Big(  \pi_1 \mathrm{CEmb}(S^2,(Y_\#,\xi_\#))\rightarrow \pi_1\mathrm{FCEmb}(S^2,(Y_\#,\xi_\#)) \Big).$$
\end{itemize}
\end{Remark}
\begin{proof}
Denote by $S_\#=e(S^2)$ the standard separating sphere. Consider the squared Dehn twist $\tau_{S_\#^+}^2$ along a parallel copy $S_{\#}^+$ of $S_{\#}$, where we assume that $S_{\#}^{+}$ is contained in $(Y,\xi)\backslash B$, where $B$ is the Darboux ball used to perform the connected sum. By the vanishing of the Euler class of $\xi$ there  exists a homotopy through formal contactomorphisms joining the identity with $\tau_{S_\#^+}^2$ (Lemma \ref{ftrivialtwist}). It follows from Eliashberg's Theorem \ref{thm:HPrincipleOT} combined with Lemma \ref{moserflemma} that we can deform this homotopy (through formal contactomorphisms) to a homotopy $\varphi_t$ through contactomorphisms with $\varphi_0=\mathrm{id}$ and $\varphi_1=\tau_{S_\#^+}^2$. This process can be done relative to an open neighbourhood of the south pole $e(s)\in (Y\#M,\xi\#\xi_{ot})$, see Remark \ref{rmk:RelativeH-Principle}.  The loop of standard spheres $\varphi_t\circ e$ is formally trivial by construction but geometrically non-trivial. Indeed, by the contact isotopy extension theorem, the triviality of this loop would imply that $\tau_{S_\#^+}^2$, regarded as a contactomorphism of $(Y,\xi)$, is contact isotopic to the identity rel. $B$, which is in contradiction with Corollaries \ref{criterionB} and \ref{cor:NoLagrangianRot}.
\end{proof}

Given a contact $3$-manifold $(Y,\xi)$ and a transverse knot $K\subset (Y,\xi)$ one can replace a small tubular neighbourhood of $K$ by a \textit{Lutz Twist} $(LT =\mathbb{D}^2\times S^1,\xi_{\mathrm{ot}})$ to obtain an \textit{overtwisted} contact manifold $(Y,\xi_K)$. Intuitively, the Lutz Twist $(LT ,\xi_{\mathrm{ot}})$ is an \textit{embedded} $S^1$-family of overtwisted disks, see \cite{geiges} for the precise definitions. We will denote by $\mathrm{LT}(Y,\xi_K)$ the space of contact embeddings $e: (LT,\xi_{\mathrm{ot}})\hookrightarrow (Y,\xi_K)$, regarded as a based space with basepoint the standard one, and by $\mathrm{FLT}(Y,\xi_K)$ the corresponding space of formal contact embeddings. As before, there is an inclusion map 
$\mathrm{LT}(Y,\xi_K)\rightarrow \mathrm{FLT}(Y,\xi_K)$. The following can be deduced following using the same strategy as above:

\begin{Corollary}\label{cor:RigidityInOTDehnTwists}
Let $(Y,\xi)$ be a irreducible contact $3$-manifold with vanishing Euler class and such that $\mathbf{c}(\xi) \notin \mathrm{Im} U$. Consider a Darboux ball $B\subset (Y,\xi)$ and a transverse knot $K\subset B$. Then, there exists an element with infinite order in $$\mathrm{Ker}\Big( \pi_1 \mathrm{LT}(Y,\xi_K)\rightarrow \pi_1\mathrm{FLT}(Y,\xi_K) \Big).$$
\end{Corollary}

\printbibliography

\end{document}